\documentclass[reqno]{amsart}

\usepackage{amsmath, amssymb, amsthm}
\usepackage{mathtools}
\usepackage[utf8]{inputenc}
\usepackage{booktabs}
\usepackage{graphicx}
\usepackage{xcolor}
\usepackage{hyperref}
\hypersetup{colorlinks=true,linkcolor={blue!50!black},citecolor={green!40!black},urlcolor={blue!60!black}}
\usepackage[capitalize,nameinlink]{cleveref}
\usepackage{caption}

\newcommand{\M}{\mathcal{M}}
\newcommand{\Q}{\mathbb{Q}}
\newcommand{\Z}{\mathbb{Z}}

\newcommand{\R}{\mathbb{R}}

\renewcommand{\P}{\mathbb{P}}
\newcommand{\Aut}{\operatorname{Aut}}
\newcommand{\Br}{\operatorname{Br}}
\newcommand{\inv}{\operatorname{inv}}
\newcommand{\Res}{\operatorname{Res}}
\newcommand{\Rat}{\operatorname{Rat}}
\newcommand{\PGL}{\operatorname{PGL}}
\newcommand{\SL}{\operatorname{SL}}
\newcommand{\GL}{\operatorname{GL}}
\newcommand{\SO}{\operatorname{SO}}
\newcommand{\Lo}[1]{\mathcal{L}_{#1}}
\newcommand{\Nf}{N_{\mathrm{fine}}}
\newcommand{\Xb}{\overline{X}}
\newcommand{\mm}{m_{12}}
\newcommand{\w}{\mathbf{w}}

\newtheorem{thm}{Theorem}[section]      			
\newtheorem{prop}[thm]{Proposition}		
\newtheorem{lem}[thm]{Lemma}				
\newtheorem{cor}[thm]{Corollary}			
\newtheorem{conj}[thm]{Conjecture}			

\newtheorem*{conjecture*}{Conjecture}
\theoremstyle{definition}
\newtheorem{defn}[thm]{Definition}			
\newtheorem{exa}[thm]{Example}			
\newtheorem{rem}[thm]{Remark}				
\newtheorem{alg}[thm]{Algorithm}			

\crefname{thm}{Thm.}{Thms.}
\crefname{prop}{Prop.}{Props.}
\crefname{lem}{Lem.}{Lems.}
\crefname{cor}{Cor.}{Cors.}
\crefname{conj}{Conj.}{Conjs.}
\crefname{defn}{Def.}{Defs.}
\crefname{exa}{Ex.}{Exs.}
\crefname{rem}{Rem.}{Rems.}
\crefname{alg}{Alg.}{Algs.}
\crefname{equation}{Eq.}{Eqs.}
\crefname{section}{Sec.}{Secs.}
\crefname{figure}{Fig.}{Figs.}
\crefname{table}{Table}{Tables}

\providecommand{\Nseven}{$214\,260\,288$}

\usepackage{geometry}

\title{Fine points of the moduli space of cubic rational maps}

\author{Tanush Shaska}
\address{Department of Mathematics and Statistics, Oakland University, Rochester, MI, 48309}
\email{shaska@oakland.edu}

\date{\today}

\begin{document}

\begin{abstract}
A rational point of the moduli space $\M_3$ of cubic rational maps is \emph{fine} if its conjugacy class contains a map
defined over $\Q$, and \emph{coarse} otherwise. In the invariants $\xi_0,\dots,\xi_5$ that embed $\M_3$ in
$\P(2,2,3,3,4,6)$, we prove that a point with trivial automorphism group is fine if and only if the quaternion algebra
$(-\mm,-2\xi_0)_\Q$ splits, where $\mm=\tfrac13\xi_0^2\xi_1+\xi_0\xi_4-\xi_3^2$; we give the analogous criterion on the
locus $\Lo2$ of maps with an involution and show that all other strata are fine. The proof is constructive and yields an
explicit descent. We prove that the obstruction class is unramified along every divisor of the compactification, that a
prime $p\ge5$ obstructs only if the point reduces modulo $p$ into the closure of one of the two codimension-two loci
$\Lo1$, $\Lo2$ of maps with an involution, and that $\M_3$ has transverse $A_1$-singularities along both, above which
the class ramifies; local densities predict that $\Lo2$ is met modulo $p$ with probability $\asymp1/p$, so that it
carries the obstruction. A census
of all \Nseven{} rational points of weighted moduli height at most $7$ agrees with this mechanism and supports the
conjecture that the proportion of fine points tends to zero like $(\log B)^{-1/2}$.
\end{abstract}

\subjclass[2020]{Primary 37P45, 14D22; Secondary 37P05, 14G05, 14F22, 11R52, 13A50, 68W30}

\keywords{rational maps, moduli spaces, field of moduli, quaternion algebras, weighted heights, arithmetic statistics}

\maketitle

\section{Introduction}

A rational map $\phi\colon\P^1\to\P^1$ of degree $d\ge2$ is determined by a pair of binary forms of degree
$d$ without common zero, and $\PGL_2$ acts on such pairs by conjugation, 
\[
\phi\mapsto\phi^\sigma=\sigma^{-1}\circ\phi\circ\sigma.
\]
The quotient $\M_d$ is a geometric quotient, an affine variety of dimension $2d-2$ \cite{levy-2011},
and a point of $\M_d(\overline\Q)$ is a conjugacy class. A point $P\in\M_d(\Q)$ is a class whose field of
moduli is $\Q$; it need not contain a map defined over $\Q$. The field of moduli is always a field of definition when
$d$ is even, and this fails in odd degree \cite[\S4.10]{silverman-2007-ads}, with the
cubic example $i\bigl(\tfrac{x-1}{x+1}\bigr)^3$. 
In the language of stacks, a $\Q$-point of $[\Rat_d/\PGL_2]$ over $P$ is a $\PGL_2$-torsor $E$ over $\Q$
with an equivariant map $E\to\Rat_d$, that is, a rational map of degree $d$ on the conic $C=E\times^{\PGL_2}\P^1$;
when $\Aut(\phi)$ is trivial there is exactly one such pair over $P$, and $\phi$ descends to $\Q$ if and only if
$C\cong\P^1$. We call $P$ \emph{fine} if $\phi$ descends to $\Q$ and \emph{coarse} otherwise.

This paper is the dynamical counterpart of our study of the rational points of the moduli space $\M_2$ of genus-two
curves ordered by weighted moduli height \cite{2024-03,2026-stats}: we ask how the fine points are distributed among the rational points ordered by height, and what the
arithmetic mechanism is. For cubic maps the moduli space is explicit. In \cite{2024-4-arxiv}  one attaches to $\phi$ its fixed-point quartic and its divergence quadratic, whose six
invariants $\xi_0,\dots,\xi_5$, of weights $2,2,3,3,4,6$, identify $\M_3$ with the complement of the resultant in a
hypersurface $\Xb\colon\xi_5^2=F(\xi_0,\dots,\xi_4)$ of $\P(2,2,3,3,4,6)$. The descent problem is governed by a conic
over the field of moduli \cite[Thms.~4.3, 4.4]{west-cubic}.

We make this explicit. For a point with trivial automorphism group, \cref{thm:criterion} shows that $P$ is fine if and
only if the quaternion algebra $(-\mm,-2\xi_0)$ splits, where $-2\xi_0$ is the discriminant of the divergence and
\[
\mm=\tfrac13\xi_0^2\xi_1+\xi_0\xi_4-\xi_3^2
\]
 is a principal minor of the Gram matrix of three covariant quadratics
(\cref{lem:gram}); on the locus $\Lo2$ of maps with an involution the algebra is $(-2\xi_0,2(\xi_0\xi_1-3\xi_4))$
(\cref{prop:L2}), and all other strata are fine. The proof is constructive: a rational point on the conic gives a map
over $\Q$ by linear algebra (\cref{prop:explicit}). Among the consequences (\cref{cor:consequences}), every point has a
model over $\Q(\sqrt{-2\xi_0})$, and the pseudo-real points of $U$ are those with $\xi_0>0$ and $\mm>0$. The maps
$(3+2\sqrt2)\frac{z^d+1}{z^d-1}$, stated in \cite[\S5]{hidalgo-quispe} not to be definable over $\Q$, are definable over
$\Q$ for every odd $d$ (\cref{ex:hq}).

On the stratum $U$ of maps with trivial automorphism group the obstruction is a class $\beta\in\Br(U)[2]$. It is unramified along every prime divisor
of $\Xb$ (\cref{thm:unram}); a prime $p\ge5$ obstructs at $P$ only if $P$ reduces modulo $p$ into the closure $Z$ of
$\Lo2$ or into the closure of $\Lo1$, and then, under the conditions of \cref{thm:local}, the local invariant is
given by the Legendre symbol of $-2\xi_0$, respectively of $-6(\xi_0\xi_1+3\xi_4)$; and $\Xb$ has transverse $A_1$-singularities along $\Lo2$
and along $\Lo1$, on whose exceptional divisors $\beta$ acquires these residues (\cref{thm:L2geom}). As in genus two, the
obstruction lives over the loci of maps with an extra involution, but here these loci have codimension two, and the
ramification appears only after a blow-up. Local densities predict that $\Lo2$ behaves like a divisor, met modulo $p$ with
probability $\asymp1/p$, while reduction into $\Lo1$ is rare, so that the obstruction comes from $\Lo2$; this leads to
\cref{conj:density}: the proportion of fine points of height at most $B$ tends to zero like $(\log B)^{-1/2}$.

The census of all \Nseven{} points of $\M_3(\Q)$ of weighted moduli height at most $7$ (\cref{sec:census,sec:statistics})
agrees with \cref{thm:local} on all $75\,546\,209$ obstructed pairs $(P,p)$ with $p\ge5$. Of its points, $68.0\%$ are fine;
the fine proportion falls geometrically in the number of primes at which the point reduces into $\Lo2$, and, by
reciprocity, $83\%$ of the coarse points are obstructed at $\infty$, $2$ or $3$. For fine points, the explicit descent
followed by minimisation and reduction (\cref{alg:small}) recovers the known minimal naive height in all but one of
$7\,409$ test cases.

Finally, in \cref{sec:orbits} we state the cubic case of the conjecture of \cite[Conj.~3.47]{silverman-2007-ads} on
integers in orbits through the moduli height (\cref{conj:integral-orbits}). Resultant minimality cannot be
dropped, even at a fixed moduli point (\cref{ex:orbit-scaling,prop:fixed-height-twists}); on the fine generic points of
bounded height the number of integers in wandering orbits of minimal maps is bounded (\cref{prop:generic-orbit-bound}),
by the results of \cite{bruin-molnar-2012,yap-2026-integral}.

\medskip
The paper is organised as follows. \cref{sec:background} recalls the invariants, the strata and the height.
\cref{sec:descent} gives the descent criterion and the explicit descent, and \cref{sec:local} the arithmetic of the
obstruction class. \cref{sec:census} describes the census and \cref{sec:statistics} its statistics, \cref{sec:models}
the minimal models and \cref{sec:orbits} integer points in orbits. For the programs, see \cite{2024-4-arxiv} and the
ancillary files of this paper.

\section{The moduli space of cubic rational maps}
\label{sec:background}

\subsection{Invariants}
Let $k$ be a field of characteristic zero and $V_n$ the space of binary forms of degree $n$ over $k$. For $f\in V_m$
and $g\in V_n$ we write $(f,g)_r$ for the $r$-th transvectant, normalised as in \cite[\S2]{2024-4-arxiv}, to which we
refer for the invariant theory used below.
A cubic map 
\[
\phi(x)=\dfrac{c_0x^3+c_1x^2+c_2x+c_3}{c_4x^3+c_5x^2+c_6x+c_7}
\]
 is the pair of binary cubics $(f_0,f_1)$,
and with \cite{west-cubic,2024-4-arxiv} we put
\[
I_\phi=yf_0-xf_1\in V_4,\qquad J_\phi=\frac{\partial f_0}{\partial x}+\frac{\partial f_1}{\partial y}\in V_2 .
\]
The roots of $I_\phi$ are the fixed points of $\phi$. The map $(f_0,f_1)\mapsto(I_\phi,J_\phi)$ is a linear bijection
from the pairs of cubics with nonzero resultant onto the pairs $(I,J)$ with nonzero modular resultant, with inverse
\[
(I,J)\mapsto\bigl(\tfrac14(xJ+\partial_yI),\tfrac14(yJ-\partial_xI)\bigr);
\]
 it identifies $\Rat_3$, the pairs
$(f_0,f_1)$ up to a common scalar, with the pairs $(I,J)$ up to a common scalar, and $\phi^\sigma$ corresponds to the
substitution action of $\sigma$ on the pair \cite{2024-4-arxiv}. The invariants are
\begin{equation}\label{eq:xi}
\begin{aligned}
\xi_0&=(J,J)_2, & \xi_1&=(I,I)_4, & \xi_2&=\bigl((I,I)_2,I\bigr)_4, \\
\xi_3&=(I,J^2)_4, & \xi_4&=\bigl((I,I)_2,J^2\bigr)_4, & \xi_5&=\bigl(J^3,(I,(I,I)_2)_1\bigr)_6,
\end{aligned}
\end{equation}
homogeneous of degrees $2,2,3,3,4,6$ in the pair, hence of the same weights under $\phi\mapsto\phi^\sigma$,
$\sigma\in\GL_2$. They generate the ring of invariants of $V_4\oplus V_2$, subject to the single relation
\cite[eq.~(15)]{west-cubic}, \cite{2024-4-arxiv}
\begin{equation}\label{eq:rel}
\xi_5^2=F(\xi_0,\dots,\xi_4)=\tfrac1{108}\xi_0^3\xi_1^3-\tfrac1{18}\xi_0^3\xi_2^2-\tfrac1{24}\xi_0\xi_1^2\xi_3^2-\tfrac16\xi_2\xi_3^3
+\tfrac12\xi_0\xi_2\xi_3\xi_4+\tfrac14\xi_1\xi_3^2\xi_4-\tfrac14\xi_0\xi_1\xi_4^2-\tfrac12\xi_4^3,
\end{equation}
and the resultant of $f_0,f_1$ is
\begin{equation}\label{eq:I6}
I_6=\tfrac18\xi_1^3+\tfrac1{384}\xi_0^2\xi_1-\tfrac34\xi_2^2+\tfrac3{16}\xi_1\xi_4+\tfrac1{256}\xi_3^2
-\tfrac3{16}\xi_2\xi_3-\tfrac1{64}\xi_0\xi_4+\tfrac18\xi_5 .
\end{equation}
Thus $\M_3\cong\Xb\setminus\{I_6=0\}$ with $\Xb\subset\P(2,2,3,3,4,6)$ the hypersurface \cref{eq:rel}, and two
maps are conjugate over $\overline k$ if and only if their invariants agree up to the weighted scaling
$\xi_i\mapsto\lambda^{w_i}\xi_i$, $\w=(2,2,3,3,4,6)$ \cite{2024-4-arxiv}. The relation \cref{eq:rel} is
quadratic in $\xi_5$, and $\Xb$ is a double cover of $\P(2,2,3,3,4)$ branched along $F=0$; for $\xi_5\ne0$ the two
points $(\xi_0,\dots,\xi_4,\pm\xi_5)$ are the two conjugacy classes with the same first five invariants, except at
the single point $[0:0:0:0:-2:2]$, where the scaling $\lambda=i$ identifies them.

\subsection{Rational points and the height}
\begin{lem}\label{lem:ratrep}
Every point of $\M_3(k)$ has a representative $(\xi_0,\dots,\xi_5)\in k^6$. If the weights of its nonzero
coordinates have greatest common divisor $g$, two $k$-rational representatives differ by $\xi_i\mapsto
\mu^{w_i/g}\xi_i$ with $\mu\in k^\times$.
\end{lem}
\begin{proof}
Let $\xi$ be any representative over $\overline k$ and $S$ its support. Choose integers $e_i$, $i\in S$, with
$\sum e_iw_i=g$ and put $t=\prod_{i\in S}\xi_i^{e_i}$, of weight $g$. Since $P$ is fixed by Galois, each
$\tau\in\operatorname{Gal}(\overline k/k)$ sends $\xi$ to $\lambda_\tau^{\w}\xi$ for some $\lambda_\tau\in\overline k^\times$,
hence $t$ to $\lambda_\tau^gt$, and the weight-zero quotients $\xi'_i=\xi_i/t^{w_i/g}$ are fixed by $\tau$; so
$\xi'$ is a $k$-rational representative. The
second statement is the same computation applied to the ratio of two representatives.
\end{proof}

For $k=\Q$ every point therefore has an integral representative which is primitive for the effective weights
$w_i/g$ on its support, unique up to the sign $\mu=-1$; we fix the sign by asking that the first nonzero coordinate of odd
effective weight be positive, and call the result the \emph{canonical representative}. For points of trivial
support gcd, which include all points of the generic stratum, this means: no prime $p$ with $p^2\mid\xi_0,\xi_1$,
$p^3\mid\xi_2,\xi_3$, $p^4\mid\xi_4$, $p^6\mid\xi_5$, and the first nonzero of $\xi_2,\xi_3$ positive.

\begin{defn}\label{def:height}
The \emph{weighted moduli height} of $P\in\M_3(\Q)$ is $H(P)=\max_i|\xi_i|^{1/w_i}$, computed on the
canonical representative.
\end{defn}

Thus $H(P)\le B$ if and only if the canonical representative lies in the box $|\xi_0|,|\xi_1|\le B^2$,
$|\xi_2|,|\xi_3|\le B^3$, $|\xi_4|\le B^4$, $|\xi_5|\le B^6$, and there are finitely many such points. It is the
height used in \cite{2024-03,2026-stats} for $\M_2$; it is at least the product-formula height
$\prod_v\max_i|\xi_i|_v^{1/w_i}$ of \cite{2019-1,2022-1}, with equality when the coordinates of the canonical
representative are coprime. The box contains 
\[
(2B^2+1)^2(2B^3+1)^2(2B^4+1)\approx32B^{14}
\]
 tuples
$(\xi_0,\dots,\xi_4)$, and for each of them the relation leaves at most two values of $\xi_5$; for a tuple of
size $B$ the right side of \cref{eq:rel} is of size $B^{12}$ and is a square with probability about $B^{-6}$, so
one expects $N(B)=\#\{P\in\M_3(\Q):H(P)\le B\}$ to be of order $B^{8}$, up to a logarithmic factor that we
discuss in \cref{sec:counting}.

\subsection{Automorphism strata}
The automorphism group of a cubic map is one of $1$, $C_2$ (in two ways), $C_3$, $V_4$ (in two ways), $A_4$ and
the dihedral group $D_4$ of order $8$ \cite{west-cubic,2024-4-arxiv}. \cref{tab:strata} lists the strata in the
notation of \cite{2024-4-arxiv}, with the names of \cite{west-cubic} in parentheses. We write $\Lo1,\dots,\Lo7$ for the strata minus the more special
ones, as in \cite{2024-4-arxiv}, and $U$ for the generic stratum.

\begin{table}[ht]
\centering
\caption{Automorphism strata of $\M_3$ \cite{west-cubic,2024-4-arxiv}; $\mm=\frac13\xi_0^2\xi_1+\xi_0\xi_4-\xi_3^2$.
The last two columns are the counts in the census $H\le7$.}
\label{tab:strata}
\begin{tabular}{llclrr}
\toprule
stratum & group & dim & equations & points & fine \\
\midrule
$U$ & $1$ & $4$ & --- & $214\,073\,758$ & $145\,448\,181$\\
$\Lo1$ & $C_2$ ($C_2^{(2)}$) & $2$ & $\xi_2=\xi_3=0$ & $135\,917$ & $135\,917$\\
$\Lo2$ & $C_2$ ($C_2^{(1)}$) & $2$ & $\xi_5=\mm=0$ & $16\,162$ & $3\,544$\\
$\Lo3$ & $C_3$ & $1$ & $\xi_0=\xi_1=\xi_4=0$ & $327$ & $327$\\
$\Lo4$ & $V_4$ ($D_4^{(1)}$) & $1$ & $\xi_2=\xi_3=\xi_5=0,\ \xi_0\xi_1+3\xi_4=0$ & $1\,641$ & $1\,641$\\
$\Lo5$ & $V_4$ ($D_4^{(2)}$) & $1$ & $\xi_0=\xi_3=\xi_4=\xi_5=0$ & $32\,481$ & $32\,481$\\
$\Lo6$ & $A_4$ & $0$ & $[0:0:1:0:0:0]$ & $1$ & $1$\\
$\Lo7$ & $D_4$ ($D_8$) & $0$ & $[0:1:0:0:0:0]$ & $1$ & $1$\\
\bottomrule
\end{tabular}
\end{table}

The stratum $\Lo1$ is the locus where the weighted projective space has the stabiliser $\mu_2$ (the odd-weight
coordinates vanish), and $\Lo3$ the locus with stabiliser $\mu_3$; $\Lo2$, on the other hand, lies on the branch
divisor $\xi_5=0$ of $\Xb$. On $\Lo2$ the invariants are parametrised by the normal form $J=s\,xy$,
$I=tx^4+ux^2y^2+vy^4$:
\begin{equation}\label{eq:L2par}
(\xi_0,\dots,\xi_5)=\Bigl(-\tfrac{\sigma}2,\ \tfrac{12w+u^2}6,\ \tfrac{u(36w-u^2)}{36},\ \tfrac{\sigma u}6,\ \tfrac{\sigma(12w-u^2)}{36},\ 0\Bigr),
\qquad \sigma=s^2,\ w=tv,
\end{equation}
a birational map from the weighted plane $\P(2,1,2)$ with coordinates $(\sigma,u,w)$ onto $\Lo2$.

\section{The descent obstruction}
\label{sec:descent}

On $V_2$ the second transvectant is a symmetric bilinear form,
\[
\langle u,v\rangle=(u,v)_2,\qquad \langle u,u\rangle=\tfrac12(4ac-b^2)=-\tfrac12\operatorname{disc}(u)\quad(u=ax^2+bxy+cy^2),
\]
nondegenerate of discriminant $\tfrac12$, and isotropic ($\langle x^2,x^2\rangle=0$). The substitution action of
$\SL_2$ preserves it and induces the isomorphism $\PGL_2\cong\SO(V_2,\langle\,,\rangle)$; the Jacobian
$(u,v)_1$ is $\SO$-equivariant, orthogonal to $u$ and $v$, and changes sign under the improper isometries.
For a pair $(I,J)\in V_4\oplus V_2$ consider the covariant quadratics
\begin{equation}\label{eq:u}
u_1=J,\qquad u_2=(I,J)_2,\qquad u_3=(u_1,u_2)_1,\qquad v_3=\bigl((I,I)_2,J\bigr)_2,
\end{equation}
of degrees $1,2,3,3$ in the pair. Their Gram matrices are invariants, hence polynomials in the $\xi_i$.

\begin{lem}\label{lem:gram}
Put $\mm=\tfrac13\xi_0^2\xi_1+\xi_0\xi_4-\xi_3^2$. Then
\[
\widetilde G:=\bigl(\langle u_i,u_j\rangle\bigr)_{i,j\le3}=\begin{pmatrix}\xi_0&\xi_3&0\\ \xi_3&\tfrac13\xi_0\xi_1+\xi_4&0\\ 0&0&\tfrac12\mm\end{pmatrix},
\]
and, writing $G$ for the Gram matrix of $(u_1,u_2,v_3)$,
\[
G=\begin{pmatrix}\xi_0&\xi_3&\xi_4\\ \xi_3&\tfrac13\xi_0\xi_1+\xi_4&\tfrac13\xi_0\xi_2+\tfrac16\xi_1\xi_3\\ \xi_4&\tfrac13\xi_0\xi_2+\tfrac16\xi_1\xi_3&\tfrac1{18}\xi_0\xi_1^2-\tfrac16\xi_1\xi_4+\tfrac13\xi_2\xi_3\end{pmatrix}.
\]
Moreover $\det G=2F(\xi_0,\dots,\xi_4)$, so that $\det G=2\xi_5^2$ on $\Xb$, and the determinant of the coefficient
matrix of $(u_1,u_2,v_3)$ is $2\xi_5$. Finally, with $T_{kl}=(I,u_ku_l)_4$,
\begin{equation}\label{eq:T}
\begin{gathered}
T_{11}=\xi_3,\quad T_{12}=\tfrac13\xi_0\xi_1+\xi_4,\quad T_{13}=0,\quad T_{22}=\tfrac13\xi_0\xi_2+\tfrac12\xi_1\xi_3,\\
T_{23}=\xi_5,\quad T_{33}=-\tfrac16\xi_0^2\xi_2-\tfrac1{12}\xi_0\xi_1\xi_3+\tfrac12\xi_3\xi_4 .
\end{gathered}
\end{equation}
\end{lem}

\begin{proof}
Each entry is an invariant of $V_4\oplus V_2$ of the indicated weight ($\langle u_i,u_j\rangle$ has weight
$\deg u_i+\deg u_j$, and $T_{kl}$ weight $1+k+l$), hence a polynomial in $\xi_0,\dots,\xi_5$ with rational
coefficients. We determined these polynomials by solving for the coefficients of all monomials of the given
weight on random evaluations, and then checked each identity symbolically on the generic pair
\[
(I,J)=(\sum a_ix^iy^{4-i},\sum b_jx^jy^{2-j}).
\]
 That $\langle u_1,u_3\rangle=\langle u_2,u_3\rangle=0$ is the orthogonality
of the Jacobian, and 
\[
\langle u_3,u_3\rangle=\tfrac12\det\bigl(\langle u_i,u_j\rangle\bigr)_{i,j\le2}
\]
 is the
classical identity for the Jacobian of two binary quadratics. The identity $\det G=2F$ is equivalent to the
relation \cref{eq:rel}, and was checked in the same way.
\end{proof}

We record the identity, valid for any symmetric $3\times3$ matrix,
\begin{equation}\label{eq:sylv}
G_{11}\det G=\mm\, n-h^2,\qquad h=G_{11}G_{23}-G_{12}G_{13},\quad n=G_{11}G_{33}-G_{13}^2,
\end{equation}
(indices from $1$), so that on $\Xb$
\begin{equation}\label{eq:sylv2}
2\xi_0\xi_5^2=\mm\,n-h^2 .
\end{equation}
In particular on $\{\mm=0\}$ we have $-2\xi_0\xi_5^2=h^2$, so that $h$ vanishes on $\{\xi_5=\mm=0\}$, and 
$\{\xi_5=\mm=0\}=\{\mm=h=0\}$ on $\Xb\cap\{\xi_0\ne0\}$ (on $\xi_0=0$ the second locus is larger: it contains the surface $\{\xi_0=\xi_3=0\}\cap\Xb$, for instance $[0:0:1:0:-2:2]$); the locus $\{\xi_5=\mm=0\}$ is $\Lo2$ of \cref{tab:strata} together with the more special strata in its closure \cite{west-cubic,2024-4-arxiv}. Geometrically, $\mm=0$ says that the plane spanned by $u_1,u_2$ is degenerate for $\langle\,,\rangle$, i.e.\ that $u_1$ and $u_2$ have a common root. 
We write $Z=\{\operatorname{rank}G\le1\}\subset\Xb$.

\begin{lem}\label{lem:Z}
$Z=\{\xi_5=\mm=0\}\cap\Xb$ is irreducible of dimension two; it is the closure of $\Lo2$, and contains $\Lo4$, $\Lo5$, $\Lo6$ and $\Lo7$.
\end{lem}
\begin{proof}
Since $\mm$ is a principal minor of $G$ and $\det G=2\xi_5^2$, $\operatorname{rank}G\le1$ implies $\xi_5=\mm=0$; and on $\{\xi_5=\mm=0\}$ the identity \cref{eq:sylv2} gives $h=0$. Where $\xi_0\ne0$, the equations $\mm=0$ and $h=0$ are solved by $\xi_4=\xi_3^2/\xi_0-\tfrac13\xi_0\xi_1$ and $\xi_2=3\xi_3^3/\xi_0^3-3\xi_1\xi_3/(2\xi_0)$, so $\{\mm=h=0,\ \xi_0\ne0\}$ is the graph of a map of the coordinates $(\xi_0,\xi_1,\xi_3)$, an irreducible surface $S^\circ$; by \cref{eq:sylv} it lies on $\{F=0\}$, hence on $\Xb$ with $\xi_5=0$. The image of \cref{eq:L2par} is a two-dimensional subset of $S^\circ$, hence dense in it, and $G$ has rank one on it; since $\operatorname{rank}G\le1$ is a closed condition, it holds on all of $S^\circ$. Where $\xi_0=0$, $\mm=0$ forces $\xi_3=0$, and then $F=-\tfrac12\xi_4^3$ forces $\xi_4=0$: this is the locus $\{\xi_0=\xi_3=\xi_4=\xi_5=0\}$ of $\Lo5$, $\Lo6$, $\Lo7$, on which $G=0$. So $\{\xi_5=\mm=0\}\cap\Xb$ is the union of the closure $S$ of $S^\circ$ and of this curve; since it is cut out of the fourfold $\Xb$ by two equations, each of its components has dimension at least two, so the curve lies in $S$, and $Z=\{\xi_5=\mm=0\}\cap\Xb=S$ is irreducible. Finally $\Lo2$ is the image of \cref{eq:L2par}, which is dense in $S$.
\end{proof}


\begin{thm}\label{thm:criterion}
Let $k$ be a field of characteristic zero, $P\in U(k)$ a point with trivial automorphism group, and
$(\xi_0,\dots,\xi_5)\in k^6$ a representative (\cref{lem:ratrep}).
\begin{enumerate}
\item If $\mm\neq0$ and $\xi_0\ne0$, then $P$ is fine if and only if the quaternion algebra $(-\mm,-2\xi_0)_k$ is
split, i.e.\ if and only if the conic $\xi_0X^2+\frac{\mm}{\xi_0}Y^2+\frac{\mm}2Z^2=0$ has a $k$-point.
\item If $\mm=0$ or $\xi_0=0$, then $P$ is fine.
\end{enumerate}
The class $(-\mm,-2\xi_0)$ does not depend on the representative, since $\xi_0$ and $\mm$ have even weights.
\end{thm}

\begin{proof}
Points of $U$ have a nonzero coordinate of weight $3$ (otherwise $P\in\Lo1$) and one of weight $2$ or $4$
(otherwise $P\in\Lo3$), so by \cref{lem:ratrep} the $k$-rational representatives of $P$ are the
$(\lambda^{w_i}\xi_i)$ with $\lambda\in k^\times$, and replacing $(I,J)$ by $(\mu I,\mu J)$ does not change the map
while multiplying the invariants by $\mu^{\w}$. Hence $P$ is fine if and only if there is a pair
$(I,J)\in(V_4\oplus V_2)(k)$, defining a map, whose invariants are exactly $\xi$.

Suppose first $\mm\ne0$. Then $\det\widetilde G=\tfrac12\mm^2\ne0$, so $u_1,u_2,u_3$ is a basis of $V_2$ for every
pair with invariants $\xi$. If $(I,J)$ is defined over $k$, so are the $u_i$, and $(k^3,\widetilde G)\cong(V_2,\langle\,,\rangle)$
over $k$. Conversely assume such an isometry exists, and let $u'_1,u'_2\in V_2(k)$ be the images of the first two
basis vectors, so that $\langle u'_i,u'_j\rangle=\widetilde G_{ij}$ for $i,j\le2$; put $u'_3=(u'_1,u'_2)_1$. Let
$(I_0,J_0)$ be any pair over $\overline k$ with invariants $\xi$ and $u_i$ its covariants. By Witt's theorem the isometry $\operatorname{span}(u_1,u_2)\to\operatorname{span}(u'_1,u'_2)$, $u_i\mapsto u'_i$, extends to an isometry of $V_2$ in exactly two ways, which differ by the reflection in the orthogonal line; one of them, $g$, is proper. Then $g$ is the substitution by some $\sigma\in\SL_2(\overline k)$, it maps the Jacobian $u_3$ to $u'_3$, and $(I,J)=(I_0,J_0)^\sigma$ has invariants $\xi$, $J=u'_1$ and $(I,J)_2=u'_2$. So $J$ and $u_1,u_2,u_3$ are defined over $k$. Now $I$ is determined by the six numbers
$(I,u_ku_l)_4=T_{kl}(\xi)\in k$: the products $u_ku_l$ span $V_4$ and $(\,,\,)_4$ is a nondegenerate pairing on
$V_4$. The linear system $(I,u_ku_l)_4=T_{kl}(\xi)$ has coefficients in $k$ and a unique solution over
$\overline k$, which therefore lies in $V_4(k)$. So $(I,J)$ is defined over $k$, and $P$ is fine.

The forms $\widetilde G$ and $\langle\,,\rangle$ have the same discriminant ($\tfrac12\mm^2\equiv\tfrac12$ modulo
squares), and $\langle\,,\rangle$ is isotropic. Two ternary forms of the same discriminant, one isotropic, are
isometric if and only if the other is isotropic (both are then $\mathbb H\perp\langle-\mathrm{disc}\rangle$). Thus $P$
is fine if and only if $\widetilde G$ is isotropic over $k$. If $\xi_0=0$ the first basis vector is isotropic. If
$\xi_0\ne0$, $\widetilde G\cong\langle\xi_0,\mm/\xi_0,\mm/2\rangle$, which is isotropic if and only if
$\bigl(-\mm/\xi_0^2,\,-\mm/(2\xi_0)\bigr)=(-\mm,-2\xi_0\mm)=(-\mm,-2\xi_0)$ is split, using $(a,-a)=0$.

Suppose now $\mm=0$. Since $P\notin\Lo2$, \cref{eq:sylv2} gives $\xi_5\ne0$, so by \cref{lem:gram} the
covariants $u_1,u_2,v_3$ form a basis of $V_2$ with Gram matrix $G$. The argument above applies verbatim to this
basis (with $-g$ in place of $g$ when $\det g=-1$, which now replaces all three vectors by their negatives,
again $k$-rational), the pairings $(I,u_ku_l)_4$, $(I,u_kv_3)_4$, $(I,v_3^2)_4$ being invariants and therefore in
$k$. So $P$ is fine if and only if $G$ is isotropic. But the upper left $2\times2$ block of $G$ has determinant
$\mm=0$, so it has a nonzero $k$-rational kernel vector, which is isotropic for $G$.
\end{proof}

\begin{rem}\label{rem:west}
The triple $(u_1,u_2,u_3)$ is the triple $(f,(g,f)_2,((g,f)_2,f)_1)$ of \cite[Thm.~4.4]{west-cubic} (where $f$ is our
$J$ and $g$ our $I$), and the invariant $\tilde c$ there is $\tfrac12\mm$; the triple $(u_1,u_2,v_3)$ is that of
\cite[Thm.~4.3]{west-cubic}, whose invariant $c$ is a multiple of $\xi_5$. The content of \cref{thm:criterion} beyond
these theorems is that the conic is diagonal up to a $2\times2$ block, so that the obstruction is the explicit symbol
$(-\mm,-2\xi_0)$, and the proof shows that the descended map can be computed by linear algebra once a point on the
conic is known (\cref{prop:explicit}).
\end{rem}

For the strata with nontrivial automorphisms we rely on \cite[\S4.3, Prop.~4.5]{west-cubic}. On
$\Lo1$ (where the weighted projective space has stabiliser $\mu_2$) and on $\Lo3$ (stabiliser $\mu_3$) the
rational representatives are only determined up to $\mu^{w_i/2}$, respectively $\mu^{w_i/3}$, with $\mu\in k^\times$,
and this freedom is enough to make the conics of \cite{west-cubic} split: every point of $\Lo1$, $\Lo3$, and of $\Lo4\subset\Lo1$, is
fine. The same holds on $\Lo5$ and at the $D_4$-point $[0:1:0:0:0:0]$ \cite[Prop.~4.5]{west-cubic}. The
discussion of the $A_4$-point $[0:0:1:0:0:0]$ in \cite{west-cubic} only exhibits a non-trivial class in its obstruction set; it is
nevertheless fine: the map $-3z/(4z^3+1)$, with $J=0$ and the equianharmonic fixed-point quartic $I=x^4+xy^3$, has invariants $(0,0,24,0,0,0)$. On $\Lo2$,
$P$ has a model over $k$ if and only if the conic
$9\xi_0^3X^2+8\xi_0^2Y^2-24\xi_0\xi_3YZ+(72\xi_3^2-36\xi_0^2\xi_1)Z^2=0$ has a $k$-point \cite[Prop.~4.5]{west-cubic}. Completing the square and using
$\mm=0$, i.e.\ $3\xi_3^2=\xi_0^2\xi_1+3\xi_0\xi_4$, this conic is $\langle\xi_0,\,2,\,2\xi_0(3\xi_4-\xi_0\xi_1)\rangle$. We
obtain:

\begin{prop}\label{prop:L2}
A point $P\in\Lo2(k)$ is fine if and only if $\bigl(-2\xi_0,\,2(\xi_0\xi_1-3\xi_4)\bigr)_k$ is split. In the coordinates
\cref{eq:L2par} this is the symbol $(\sigma,w)$: $P$ is fine if and only if $w=tv$ is a norm from
$k(\sqrt\sigma)=k(s)$.
\end{prop}

Here $w\ne0$ on $\Lo2\cap\M_3$: for $v=0$ the forms $f_0,f_1$ of the normal form have the common factor $x$, and for $t=0$ the common factor $y$.

Of the $16\,162$ points of $\Lo2$ of height at most $7$, $3\,544$ are fine; all $5\,908$ classes of $\Lo2$ in the database
of \cite{2024-4-arxiv} (of naive height at most $5$) satisfy the criterion. Points with automorphisms are rare
($186\,530$ of the $214\,260\,288$ points of height at most $7$), and we concentrate on $U$ from now on.

\begin{cor}\label{cor:consequences}
Let $P\in U(k)$.
\begin{enumerate}
\item $P$ has a model over $k(\sqrt{-2\xi_0})$ and over $k(\sqrt{-\mm})$. In particular $P$ is fine if $-2\xi_0$,
the discriminant of the divergence $J_\phi$ up to a square factor, is a square in $k$ (and then the divergence of
every model of $P$ over $k$ has $k$-rational roots), and if $\mm=0$, i.e.\ if $J_\phi$ and $(I_\phi,J_\phi)_2$ have a
common root.
\item If the two points $(\xi_0,\dots,\xi_4,\pm\xi_5)$ of $\Xb$ over the same point of $\P(2,2,3,3,4)$ both lie in
$\M_3$, they are both fine or both coarse.
\item For $k=\R$: $P\in U(\R)$ is not definable over $\R$ if and only if $\xi_0>0$ and $\mm>0$, i.e.\ if and only if
$\widetilde G$ (equivalently $G$, when $\xi_5\ne0$) is positive definite. These are exactly the pseudo-real points of
\cite{hidalgo-quispe} with trivial automorphism group.
\item A quadratic field $\Q(\sqrt d)$ is a field of definition of $P\in U(\Q)$ if and only if $d$ is not a square in
$\Q_v$ for every place $v$ at which $(-\mm,-2\xi_0)$ ramifies (there is no such place if $\xi_0\mm=0$).
\end{enumerate}
\end{cor}

\begin{proof}
(1) and (2) are immediate from the shape of the symbol, which does not involve $\xi_5$; for the geometric
interpretations recall $\langle J,J\rangle=-\tfrac12\operatorname{disc}J$ and the paragraph after \cref{eq:sylv2}. For (3),
$(a,b)_\R=-1$ if and only if $a,b<0$. (4) is the standard description of the quadratic splitting fields of a
quaternion algebra.
\end{proof}

The proof of \cref{thm:criterion} is constructive, and we use it as an algorithm.

\begin{prop}[Explicit descent]\label{prop:explicit}
Let $P\in U(k)$ with $\xi_0\mm\neq0$ be fine, with representative $\xi$. Put $u_1=x^2+\tfrac12\xi_0y^2$, so that
$\langle u_1,u_1\rangle=\xi_0$; the orthogonal complement of $u_1$ is spanned by $e=xy$ and $f=x^2-\tfrac12\xi_0y^2$, with
$\langle e,e\rangle=-\tfrac12$, $\langle f,f\rangle=-\xi_0$, $\langle e,f\rangle=0$. Let $(a:b:c)$ be a $k$-point of the conic
$-\tfrac12a^2-\xi_0b^2=\frac{\mm}{\xi_0}c^2$ with $c\neq0$ (such a point exists by \cref{thm:criterion}), and put
\[
u_2=\frac{\xi_3}{\xi_0}u_1+\frac ac\,e+\frac bc\,f,\qquad u_3=(u_1,u_2)_1 .
\]
The linear system $(I,u_ku_l)_4=T_{kl}(\xi)$, $1\le k\le l\le3$, with $T_{kl}$ as in \cref{eq:T}, has a unique solution $I\in V_4(k)$, it satisfies $(I,u_1)_2=u_2$, and
\[
\phi=\frac{x\,u_1+\partial_yI}{y\,u_1-\partial_xI}
\]
is a cubic map over $k$ whose pair of cubics $\bigl(\tfrac14(xu_1+\partial_yI),\tfrac14(yu_1-\partial_xI)\bigr)$ has
invariants $\xi$.
\end{prop}

\begin{proof}
The vectors $u_1,u_2$ have Gram matrix $\bigl(\begin{smallmatrix}\xi_0&\xi_3\\ \xi_3&\frac13\xi_0\xi_1+\xi_4\end{smallmatrix}\bigr)$ by
construction, since $\langle u_2,u_2\rangle=\xi_3^2/\xi_0+\mm/\xi_0$. The rest is the proof of \cref{thm:criterion};
in particular no choice of sign is needed. In our computations the system was always consistent, as it must be.
\end{proof}

In practice we solve the conic with \texttt{qfsolve} \cite{pari}; if it returns $c=0$ the
plane $u_1^\perp$ is isotropic and $u_2$ is obtained by moving along the isotropic vector. The map produced has
large coefficients, and \cref{sec:models} explains how to make it small.

\begin{exa}[The smallest coarse points]\label{ex:small}
The nine points of height $1$ are fine, and the smallest coarse points have height $3^{1/3}$. One of them is
$P=[-1:0:3:-3:4:0]$, with $\mm=-13$ and symbol $(13,2)$, ramified at $2$ and $13$. Here $\xi_5=0$, so $P$ lies on
the branch divisor, and $13\mid\mm$ with $\bigl(\frac{-2\xi_0}{13}\bigr)=\bigl(\frac2{13}\bigr)=-1$, as in
\cref{thm:local}; $P$ has a model over $\Q(\sqrt2)$.
\end{exa}

\begin{exa}[A pseudo-real map]\label{ex:west}
The map $\psi=i\bigl(\frac{x-1}{x+1}\bigr)^3$ \cite[\S4.10]{silverman-2007-ads}, \cite{west-cubic} has field of moduli $\Q$; its canonical
representative is $(144,20,12,-288,192,6912)$, of height $12$, with $\mm=82944>0$ and $\xi_0>0$. So $\widetilde G$ is
positive definite, $\psi$ is pseudo-real, and $(-\mm,-2\xi_0)=(-1,-2)$ ramifies at $\{\infty,2\}$: $\psi$ has no model
over $\Q$, in agreement with \cite{west-cubic}, but has one over $\Q(\sqrt{-1})$, $\Q(\sqrt{-2})$ and
$\Q(\sqrt{-5})$, for instance.
\end{exa}

\begin{exa}[Maps definable over $\Q$ in every odd degree]\label{ex:hq}
In \cite[\S5]{hidalgo-quispe} it is stated that for every odd $d\ge3$ the map
$\psi_d(z)=(3+2\sqrt2)\frac{z^d+1}{z^d-1}$ has field of moduli $\Q$ but cannot be defined over $\Q$. For $d=3$ the
invariants of $\psi_3$, scaled by $\lambda=(1+\sqrt2)^{-1}$, are rational, with canonical representative
$(-72,-10,42,-1008,48,-86400)$, a point of $U(\Q)$ with $-2\xi_0=144=12^2$. By \cref{cor:consequences} it is fine.
\Cref{prop:explicit}, followed by the reduction of \cref{sec:models},
produces $\phi(z)=\frac{4z^3+3z^2+6z-3}{z^3+3z-2}$, whose divergence $J_\phi=12x(x+y)$ has rational roots, and indeed
\[
\gamma^{-1}\circ\psi_3\circ\gamma=\phi,\qquad \gamma=\begin{pmatrix}-1&1-\sqrt2\\ 1-\sqrt2&1\end{pmatrix},
\]
as one checks by direct computation (the two maps also have the same fixed-point multipliers, approximately $-9.468$, $-0.0294$
and $3.249\pm7.882i$). So $\psi_3$ is definable over $\Q$.

The same $\gamma$ works in higher odd degree: a direct computation shows that $\gamma^{-1}\circ\psi_d\circ\gamma$ has rational coefficients for $d=5,7,9$, for instance
\[\gamma^{-1}\circ\psi_5\circ\gamma=\frac{13z^5+25z^4+30z^3-10z^2+25z-11}{3z^5+5z^4+10z^3-10z^2+15z-7},\]
so the statement of \cite[\S5]{hidalgo-quispe} fails for every odd $d$. Indeed, put
$\alpha=3+2\sqrt2$, $h(u)=\frac{u+1}{u-1}$ and $r(z)=-1/z$, and let $\sigma$ be the nontrivial automorphism of
$\Q(\sqrt2)$, so that $\sigma(\alpha)=\alpha^{-1}$. For odd $d$ one has $\psi_d(-1/z)=-\alpha/h(z^d)$, hence
\[r^{-1}\circ\psi_d\circ r=\alpha^{-1}h(z^d)=\psi_d^\sigma .\]
The descent datum of $\psi_d$ is thus the rational involution $r$, whose class in $H^1(\Q,\PGL_2)$ is the quaternion
algebra $(2,-1)_\Q$, split since $1^2-2\cdot1^2=-1$; and $\gamma$ is an explicit splitting,
$\gamma^\sigma=-(1+\sqrt2)\,r\gamma$, so that $(\gamma^{-1}\psi_d\gamma)^\sigma=\gamma^{-1}r^{-2}\psi_dr^2\gamma=\gamma^{-1}\psi_d\gamma$
for every odd $d$.

\end{exa}

\begin{rem}[Checks]\label{rem:checks}
(a) For $21\,615$ random maps $\phi\in\Rat_3(\Q)$ in the generic stratum, of naive heights $1$ to $30$, the symbol
$(-\mm,-2\xi_0)$ of the moduli point is split. (b) The database of \cite{2024-4-arxiv} contains every map of naive height at
most $4$; its $4\,785\,584$ geometric conjugacy classes, recomputed from their representatives and normalised as in
\cref{sec:background}, are all fine by \cref{thm:criterion} and \cref{prop:L2}, and their distribution over the strata
reproduces \cite[Table~3]{2024-4-arxiv} exactly ($4\,779\,346$ generic, $2\,972$ in $\Lo1$, $2\,610$ in $\Lo2$, $29$, $119$,
$506$, $1$, $1$). (c) For each of the $14\,225$ fine points of the census sampled in \cref{sec:models}, \cref{prop:explicit} produced a
map over $\Q$ whose invariants are a weighted rescaling of $\xi$, and for $200$ random coarse points the conic had no
rational point, as it must.
\end{rem}

\section{The arithmetic of the obstruction class}
\label{sec:local}

Over the generic stratum $U$ the group $\PGL_2$ acts freely on $\Rat_3|_U$ and the quotient is geometric
\cite{levy-2011}, so $\Rat_3|_U\to U$ is a $\PGL_2$-torsor. Its class
\[
\beta\in H^1(U,\PGL_2)\longrightarrow\Br(U)[2]
\]
specialises at $P\in U(k)$ to the class $\beta(P)\in\Br(k)[2]$ of a quaternion algebra, and $P$ is fine if and only
if $\beta(P)=0$. By \cref{thm:criterion}, $\beta(P)=(-\mm,-2\xi_0)$ when $\xi_0\mm\ne0$, and $\beta(P)=0$ otherwise;
equivalently, the conic of $\widetilde G$ (when $\mm\ne0$) or of $G$ (when $\xi_5\ne0$) is a Severi--Brauer variety of
$\beta$. Let $K=\Q(\Xb)$ be the function field and $\beta_K$ the generic class.


\begin{thm}\label{thm:unram}
$\beta_K$ is unramified along every prime divisor of $\Xb$. In particular it is unramified along the boundary
$\{I_6=0\}$, along the branch divisor $\{\xi_5=0\}$, and along $\{\mm=0\}$ and $\{\xi_0=0\}$, where the
representative $(-\mm,-2\xi_0)$ has poles or zeros.
\end{thm}

\begin{proof}
The affine cone over $\Xb$ is the hypersurface $\xi_5^2=F$ in $\mathbb A^6$, whose singular locus
$\{\xi_5=0,\ \nabla F=0\}$ has dimension three (a Gr\"obner basis computation), so the cone is normal by Serre's
criterion, and so is $\Xb$; the local ring of $\Xb$ at the generic point of a prime divisor is therefore a discrete
valuation ring.
Let $D$ be a prime divisor of $\Xb$. If $D$ meets $\M_3$, its generic point lies in $U$, because $\M_3\setminus U$
is a union of strata of dimension at most two (\cref{tab:strata}); since $\beta$ is a class on $U$, it is
unramified along $D$. Otherwise $D$ is a component of the boundary $\{I_6=0\}\cap\Xb$. Choose a rational function
$t$ of weight one on the cone over $\Xb$ which is a unit at the generic point of $D$ (a quotient of two coordinate
monomials of weights differing by one), and let $q=\sum_{i,j}G_{ij}\,t^{-(w(i)+w(j))}X_iX_j$ with
$w(1),w(2),w(3)=1,2,3$; this is a ternary form over the local ring $\mathcal O_{\Xb,D}$ representing $\beta_K$, and
$\det q=2(\xi_5/t^6)^2$. The function $\xi_5$ does not vanish identically on $D$: on $\{I_6=0\}$ we have
$96\xi_5=96\xi_5-768I_6$, a polynomial in $\xi_0,\dots,\xi_4$, and if it vanished on $D$ then $D$ would be contained in
$\{F=0,\ I_6=0\}$, which has codimension two in $\Xb$: writing $\xi_5=-8R$ on the boundary, with $R=I_6-\tfrac18\xi_5$
of weight $6$, the boundary is the hypersurface $\{F=64R^2\}$ of $\P(2,2,3,3,4)$, and its intersection with
$\{F=0\}$ is $\{F=R=0\}$, of codimension two because $F$ is irreducible of weight $12$. So $q$ is unimodular over the discrete valuation ring
$\mathcal O_{\Xb,D}$, its conic has good reduction, and $\beta_K$ is unramified at $D$.
\end{proof}

The theorem says that the obstruction class has no ramification divisor on $\Xb$ at all. This is the essential
difference with genus two, where the class of Mestre's conic is ramified exactly along the divisor of curves
with an extra involution \cite{2026-stats}: the automorphism loci of $\M_3$ have codimension two. The
ramification has not disappeared, it has moved into the singular locus. Since $\M_3$ is rational
\cite{levy-2011} and $\beta_{\overline\Q}\ne0$ (\cref{thm:L2geom}), $\beta$ cannot extend to a smooth compactification,
and the next two results show where it lives: above the closures of the two loci $\Lo1$, $\Lo2$ of maps with an
involution. We write $\overline{\Lo1}=\{\xi_2=\xi_3=0\}\cap\Xb$ for the closure of $\Lo1$.


\begin{thm}\label{thm:local}
Let $P\in U(\Q)$ with canonical representative $\xi$, and $p\ge5$ a prime. The matrix $G(P)$ is $p$-integral.
\begin{enumerate}
\item If $\inv_p\beta(P)\ne0$, then $G(P)$ has rank at most one modulo $p$ (if $\xi_5\ne0$), respectively
$p\mid\mm$ (if $\xi_5=0$), and $P$ reduces modulo $p$ into $Z\cup\overline{\Lo1}$. It reduces into $Z$ unless $p$
divides every coordinate of $\xi$.
\item Suppose $\min_iv_p(\xi_i)/w_i=\tfrac12$, $v_p(\xi_0)=1$, and $p\nmid\xi_0'\xi_1'+3\xi_4'$, where
$\xi_0'=\xi_0/p$, $\xi_1'=\xi_1/p$, $\xi_4'=\xi_4/p^2$. Then $\inv_p\beta(P)\ne0$ if and only if
$\bigl(\frac{-6(\xi_0'\xi_1'+3\xi_4')}p\bigr)=-1$.
\item If $p\nmid\xi_0$ and $\mm\ne0$, then $\inv_p\beta(P)\ne0$ if and only if $v_p(\mm)$ is odd and $\bigl(\frac{-2\xi_0}p\bigr)=-1$.
If moreover $p\nmid\xi_5$, then $\bigl(\frac{-2\xi_0}p\bigr)=1$ whenever $v_p(\mm)$ is odd, so that
$\inv_p\beta(P)\ne0$ forces $p\mid\xi_5$.
\item $\inv_\infty\beta(P)\ne0$ if and only if $\xi_0>0$ and $\mm>0$.
\end{enumerate}
\end{thm}

\begin{proof}
The entries of $G$ have denominators dividing $18$. 

(1) Let $\xi_5\neq0$. Over $\Z_p$ the form $G$ has a Jordan
splitting $\langle a_1p^{e_1},a_2p^{e_2},a_3p^{e_3}\rangle$ with units $a_i$ and $e_1\le e_2\le e_3$, and
\[
e_1+e_2+e_3=v_p(\det G)=2v_p(\xi_5)
\]
 is even. If the rank modulo $p$ is at least two, then $e_1=e_2=0$, so $e_3$ is
even and $G\cong\langle a_1,a_2,a_3\rangle$ is unimodular, hence isotropic over $\Q_p$ (Chevalley--Warning and Hensel),
and $\inv_p\beta(P)=0$. If $\xi_5=0$, then $\mm\ne0$ and the same argument applies to $\widetilde G$, of determinant
$\tfrac12\mm^2$; its rank modulo $p$ is at most one only if $p\mid\mm$. If some coordinate of $\xi$ is a $p$-adic
unit, the reduction of $P$ modulo $p$ is the reduction of $\xi$, and in both cases it satisfies $\xi_5=\mm=0$, so it
lies in $Z$. Suppose now that $p$ divides every coordinate of $\xi$, and let $r=\min_iv_p(\xi_i)/w_i$; since $\xi$ is
canonical, $0<r<1$. The reduction of $P$ is the reduction of $(p^{-rw_i}\xi_i)_i$, computed over $\Z_p[p^r]$, in which
the coordinates with $v_p(\xi_i)/w_i>r$ vanish. If $r$ has denominator $2$, the coordinates $\xi_2,\xi_3$ of odd
weight vanish in the reduction, which therefore lies in $\overline{\Lo1}$. If $r$ has denominator $3$, the coordinates
$\xi_0,\xi_1,\xi_4$ vanish in the reduction, and $v_p(\xi_0),v_p(\xi_1)\ge\lceil2r\rceil$, $v_p(\xi_4)\ge\lceil4r\rceil$,
$v_p(\xi_3)\ge3r$, so that $v_p(\xi_0^2\xi_1)$ and $v_p(\xi_0\xi_4)$ exceed $6r$ and
\[
-\mm=\xi_3^2-\tfrac13\xi_0^2\xi_1-\xi_0\xi_4\equiv\xi_3^2\pmod{p^{6r+1}}.
\]
 If $v_p(\xi_3)=3r$, then $-\mm$ is
$p^{6r}$ times a unit which is a square modulo $p$, and the symbol $(-\mm,-2\xi_0)_p$ is trivial. If $v_p(\xi_3)>3r$,
the reduction has $\xi_3=0$, hence $\xi_5^2=-\tfrac16\xi_2\xi_3^3=0$, and it is the point $[0:0:1:0:0:0]$, which lies
in $Z$ by \cref{lem:Z}. If $r$ has denominator $4$ or $6$, only $\xi_4$, respectively $\xi_5$, survives in the
reduction, which is impossible on $\Xb$.

(2) Since $r=\tfrac12$, $v_p(\xi_1)\ge1$, $v_p(\xi_3)\ge2$ and $v_p(\xi_4)\ge2$,
so 
\[
\mm\equiv p^3\xi_0'(\tfrac13\xi_0'\xi_1'+\xi_4')\pmod{p^4}
\]
 has valuation $3$, and
\[
(-\mm,-2\xi_0)_p=\bigl(\frac{-1}p\bigr)\bigl(\frac{-\mm/p^3}p\bigr)\bigl(\frac{-2\xi_0'}p\bigr)
=\bigl(\frac{-2\xi_0'^2(\xi_0'\xi_1'+3\xi_4')/3}p\bigr)=\bigl(\frac{-6(\xi_0'\xi_1'+3\xi_4')}p\bigr).
\]
 (3) For $p\nmid2\xi_0$ the Hilbert symbol is
\[
(-\mm,-2\xi_0)_p=\bigl(\frac{-2\xi_0}p\bigr)^{v_p(\mm)}.
\]
 If $v_p(\mm)$ is odd and $p\nmid\xi_5$, the identity
\cref{eq:sylv2} gives $-2\xi_0\equiv(h/\xi_5)^2\pmod p$. 

(4) is \cref{cor:consequences}(3).
\end{proof}

Thus the obstruction set $S_P$ of a point $P\in U(\Q)$ is contained in
\[
\{\infty,2,3\}\ \cup\ \{p\ge5:\ P\bmod p\in Z\cup\overline{\Lo1}\},
\]
and the second set is the set of primes at which $P$ reduces into the closure of a locus of maps with an
involution (the reduction may lie on the boundary $I_6=0$). By (1), an obstructing
prime at which $P$ reduces into $\overline{\Lo1}\setminus Z$ divides every coordinate of the canonical representative,
and when moreover $v_p(\xi_0)=1$ and the reduction is not on $\overline{\Lo4}=Z\cap\overline{\Lo1}$, the
obstruction is read off from the reduction through the character $-6(\xi_0\xi_1+3\xi_4)$ by (2), as it is read off
through $-2\xi_0$ at the primes of reduction into $Z$ by (3). In the census all $75\,546\,209$
pairs $(P,p)$ with $P\in U$ and $p\ge5$ obstructed satisfy (1); for $12\,218$ of them $p$ divides every coordinate
of the canonical representative, and $10\,754$ of these reduce into $\overline{\Lo1}\setminus Z$, the smallest being
$[-15:10:0:25:25:250]$, of height $15^{1/2}$, at $p=5$. None of the $173\,894\,528$ pairs with $p\ge5$, $p\nmid\xi_5$ and $v_p(\mm)$ odd has
$\bigl(\frac{-2\xi_0}p\bigr)=-1$ (\cref{sec:statistics}).


\begin{thm}\label{thm:L2geom}
\begin{enumerate}
\item At the generic point of $\Lo2$ the hypersurface $\Xb$ has a transverse $A_1$-singularity: there the
function $F$ vanishes to order exactly two, and its Hessian in the two normal directions is nondegenerate, so that
étale locally $\Xb\cong\Lo2\times\{y^2=q(s,t)\}$ with $q$ a nondegenerate binary quadratic form.
\item The restriction of $-2\xi_0$ to $\Lo2$ is $\sigma$ in the coordinates \cref{eq:L2par}, and it is not a square
in the function field of $\Lo2$, even over $\overline\Q$ (after normalising its weight).
\item Let $E$ be the exceptional divisor of the blow-up of $\Xb$ along $Z$. Then $\beta$ is ramified along $E$,
with residue the class of $-2\xi_0$, which is nontrivial on $E\otimes\overline\Q$.
\item At the generic point of $\Lo1$ the variety $\Xb$ has a transverse $A_1$-singularity, the quotient of a smooth
point by $\mu_2$ acting by $-1$ on the two normal directions. Let $E_1$ be the exceptional divisor of the blow-up of
$\Xb$ along $\overline{\Lo1}$. Then $\beta$ is ramified along $E_1$, with residue the class of
$-6(\xi_0\xi_1+3\xi_4)$, which is nontrivial on $E_1\otimes\overline\Q$; its zero locus on $\overline{\Lo1}$ is
$\overline{\Lo4}$.
\end{enumerate}
\end{thm}

\begin{proof}
(1) On the parametrisation \cref{eq:L2par} one computes $F=0$ and $\nabla F=0$; the Hessian of $F$ on the affine cone
has rank two at the point $(s,t,u,v)=(3,5,7,-2)$ of \cref{eq:L2par}, in exact
arithmetic, hence at the generic point of $\Lo2$ by semicontinuity, and rank at most two everywhere on $\Lo2$,
since $\nabla F$ vanishes along $\Lo2$, and it vanishes on the tangent space of the cone over $\Lo2$, which has
dimension three, so its restriction to a normal plane is nondegenerate. (2) $u=-3\xi_3/\xi_0$ is a rational function
of weight one on $\Lo2$, and $-2\xi_0/u^2=\sigma/u^2$ is one of the two coordinates of the function field
$\overline\Q(\sigma/u^2,\,w/u^2)$ of the weighted plane $\P(2,1,2)$, hence not a square. (3) In the chart $t=st'$,
$y=sy'$ of the blow-up of $y^2=q(s,t)$ the exceptional divisor is $s=0$; $\mm$ vanishes to order one along $\Lo2$ in
the normal directions (its differential at the point of (1),
$(-\tfrac{27}4,\tfrac{27}4,0,-21,-\tfrac92)$, vanishes on the tangent space of the cone over $\Lo2$ but not on the
normal plane), so $v_E(\mm)=1$ while $-2\xi_0$ is a unit, and the tame
symbol of $(-\mm,-2\xi_0)$ along $E$ is $-2\xi_0|_E$, pulled back from $\Lo2$. Since $E\to\Lo2$ is a conic
bundle with geometrically integral fibres, $k(\Lo2)$ is algebraically closed in $k(E)$, and by (2) the class of
$-2\xi_0$ is nontrivial in $k(E)^\times/k(E)^{\times2}$.
(4) The chart $\xi_0\ne0$ of $\P(2,2,3,3,4,6)$ is the quotient of $\mathbb A^5$, with coordinates
$\xi_1,\dots,\xi_5$ and $\xi_0=1$, by $\mu_2$ acting by $-1$ on $\xi_2,\xi_3$, and $\Xb$ is there the quotient of
the hypersurface $\widetilde X=\{\xi_5^2=F(1,\xi_1,\dots,\xi_4)\}$. The fixed locus $\{\xi_2=\xi_3=0\}\cap\widetilde X$
maps onto $\overline{\Lo1}\cap\{\xi_0\ne0\}$, and $\widetilde X$ is smooth at its generic point, where $\xi_5\ne0$. So
at the generic point of $\Lo1$, $\Xb$ is étale locally $\Lo1\times\mathbb A^2/\{\pm1\}$. The weight-zero functions
$s=\xi_2^2/\xi_0^3$, $\xi_2\xi_3/\xi_0^3$, $\xi_3^2/\xi_0^3$ generate the ideal of $\overline{\Lo1}$ there,
$v_{E_1}(s)=1$, and $\mm/\xi_0^3$ is a unit along $E_1$, since $\mm|_{\Lo1}=\tfrac13\xi_0^2\xi_1+\xi_0\xi_4$ is not
identically zero. With the weight-one function $t=\xi_2/\xi_0$,
\[
\beta_K=\Bigl(-\frac{\mm}{t^6},-\frac{2\xi_0}{t^2}\Bigr)=\Bigl(-\frac{\mm}{\xi_0^3}s^{-3},-2s^{-1}\Bigr)
=\Bigl(-\frac{\mm}{\xi_0^3},-2\Bigr)+\Bigl(s,-\frac{2\mm}{\xi_0^3}\Bigr),
\]
by bilinearity and $(s,s)=(s,-1)$. The first symbol is unramified along $E_1$, so the residue of $\beta_K$ along
$E_1$ is the class of $-2\mm/\xi_0^3=-\tfrac23(\xi_0\xi_1+3\xi_4)/\xi_0^2$, pulled back from $\overline{\Lo1}$, which
is the class of $-6(\xi_0\xi_1+3\xi_4)$. On $\overline{\Lo1}$ the relation \cref{eq:rel} reads
$\xi_5^2=\tfrac1{108}(\xi_0\xi_1-6\xi_4)(\xi_0\xi_1+3\xi_4)^2$; in the weight-zero coordinates $r=\xi_1/\xi_0$,
$c=\xi_4/\xi_0^2$, $d=\xi_5/\xi_0^3$ and $h=6d/(r+3c)$ one has $r=3h^2+6c$, so the function field of
$\overline{\Lo1}$ is $\overline\Q(h,c)$ and the residue is the class of $-2(h^2+3c)$, which is not a square. Since
$E_1\to\overline{\Lo1}$ is a conic bundle with geometrically integral fibres, the class stays nontrivial in
$k(E_1)^\times/k(E_1)^{\times2}$. Finally $\xi_0\xi_1+3\xi_4=0$ forces $\xi_5=0$ on $\overline{\Lo1}$, so the
zero locus is $\{\xi_2=\xi_3=\xi_5=0,\ \xi_0\xi_1+3\xi_4=0\}=\overline{\Lo4}$.
\end{proof}

\begin{rem}
Together, \cref{thm:unram,thm:local,thm:L2geom} show the same pattern as for genus two: the
obstruction lives above the loci of objects with an extra involution, and the residue character is a square
class on that locus; for $\M_2$ it is the minor $-m_{12}$ of Mestre's conic restricted to the bielliptic locus,
here it is the discriminant $-2\xi_0$ of the divergence form above $\Lo2$ and the character $-6(\xi_0\xi_1+3\xi_4)$
above $\Lo1$, each vanishing exactly on the closure of the $V_4$-locus ($\Lo5$, respectively $\Lo4$) contained in
the stratum.
The difference is that here these loci have codimension two and the ramification is only visible after a
blow-up. The points of $\Lo1$ themselves are fine: over $\Lo1$ the map $\Rat_3\to\M_3$ is not a $\PGL_2$-torsor,
and the obstruction of \cref{thm:criterion} does not apply.
\end{rem}

\label{sec:counting}

\cref{thm:L2geom}(1) has a counting consequence. A rational point of $\Xb$ is a point $x$ of $\P(2,2,3,3,4)$ at which
$F(x)$ is a square, and the $p$-adic density of such points near $x$ is governed by the number of solutions of
$y^2\equiv F(x)\pmod{p^k}$. If $x\bmod p$ lies on $\Lo2$, which happens with probability $\asymp p^{-2}$ for the
points of the box, then $p^2\mid F(x)$ and the number of square roots is multiplied by $p$, since $y=py'$ with
$y'^2=q(s',t')$ and $q$ nondegenerate. Hence a rational point of $\Xb$ reduces into $\Lo2$ modulo $p$ with probability
\[
\Pr\bigl(P\bmod p\in Z\bigr)\approx\frac{\tau_p}p,\qquad \tau_p\to1,
\]
as if $\Lo2$ were a divisor. Given this, $v_p(\mm)$ is odd with probability $1-O(1/p)$, and by
\cref{thm:L2geom}(2) and the Lang--Weil estimate on $\Lo2$ the residue character $\bigl(\frac{-2\xi_0}p\bigr)$ takes the
values $\pm1$ with probability $\tfrac12+O(p^{-1/2})$. So each prime $p\ge5$ obstructs with probability
$\approx1/(2p)$, essentially only through the event $P\bmod p\in Z$: the other route of \cref{thm:local}, through
$\overline{\Lo1}$, requires $p\mid\xi_0,\xi_1$, $p^2\mid\xi_2,\xi_3,\xi_4$ and $p^3\mid\xi_5$ for the canonical
representative, and the points of height at most $B$ with this property are expected to number $\asymp B^8/p^5$
(the box contains $\asymp B^{14}/p^8$ such tuples $(\xi_0,\dots,\xi_4)$, on which $F$ is $p^6$ times a number of
size $B^{12}/p^6$), a negligible proportion. Writing $\omega(P)$ for the number of primes $p\ge5$ at which $G(P)$ has rank at most one modulo $p$
(for the primes not dividing every coordinate, the number of primes with $P\bmod p\in Z$), and $q_p$ for the
probability that such a prime obstructs, which by the above tends to $\tfrac12$, the probabilistic model
of \cite{2026-stats} becomes
\begin{equation}\label{eq:model}
\Pr(P\text{ fine})\approx c_0(P)\prod_{p\ge5,\ \operatorname{rank}(G(P)\bmod p)\le1}(1-q_p)\approx c_0(P)\,2^{-\omega(P)},
\end{equation}
with $c_0(P)$ accounting for $\infty$, $2$, $3$ and the product formula. If, as the independence of the events
$P\bmod p\in Z$ suggests, $\omega(P)$ is distributed over the points of height at most $B$ like a Poisson variable
of mean $\lambda=\log\log B+O(1)$, then the average of $2^{-\omega}$ is $e^{-\lambda/2}\asymp(\log B)^{-1/2}$.

\begin{conj}\label{conj:density}
Let $N(B)$ and $\Nf(B)$ be the numbers of points, respectively fine points, of $\M_3(\Q)$ of height at most $B$. Then
$\Nf(B)/N(B)\to0$; more precisely $\Nf(B)/N(B)\asymp(\log B)^{-1/2}$.
\end{conj}

The same enhancement explains why the number of points itself carries a logarithm. The expected number of
rational points is a sum over the box of $\prod_p\sigma_p(x)/\sqrt{F(x)}$, and the average of the local factor
$\sigma_p$ is $1+\tau_p/p+O(p^{-2})$ because of the points near $\Lo2$; the product over $p$ diverges like
$\log B$. We therefore expect $N(B)\asymp B^8\log B$, and the census is consistent with this, although
the range $2\le B\le7$ cannot distinguish a power of $\log B$: $N(B)/(B^8\log B)$ equals $19.7$,
$17.3$, $19.8$, $20.2$, $19.1$ for $B=3,4,5,6,7$ (\cref{tab:census}). We note two heuristic lower bounds. Rational maps of naive height at most
$T$ have invariants $|\xi_i|\ll T^{w_i}$, hence moduli height $\ll T$, and there should be $\asymp T^8$ conjugacy classes
among them (the database of \cite{2024-4-arxiv} has $4\,785\,584$ classes among $20\,548\,816$ maps with $T=4$); so
$\Nf(B)\gg B^8$ and, if $N(B)\asymp B^8\log B$, the fine proportion is at least of order $(\log B)^{-1}$. The
conjecture asserts that the truth is halfway between.

A rigorous proof would require a sieve on the rational points of the double cover $\Xb$, for which, unlike for the
weighted projective space underlying $\M_2$, not even the order of growth of $N(B)$ is known.

\section{The census}
\label{sec:census}

The census consists of all points of $\M_3(\Q)$ of weighted moduli height at most $7$, each
given by its canonical representative and decorated with its automorphism stratum, its height, the invariant
$\mm$, the obstruction set $S_P$ of $\beta(P)$ and the set of primes $p\ge5$ at which $G(P)$ has rank at most one
modulo $p$. It is computed directly in the coordinates $\xi$: for every $(\xi_0,\xi_1,\xi_2,\xi_3)$ in the box, with
$\xi_2>0$ or $\xi_2=0\le\xi_3$ (one representative of each orbit of $\lambda=-1$, which changes the signs of $\xi_2,\xi_3$),
the polynomial $216F$ is a cubic in $\xi_4$ with leading coefficient $-108$, evaluated along $|\xi_4|\le B^4$ in
$64$-bit (for the coefficients) and $128$-bit (for the value) integer arithmetic, and a value is kept when it is
$216$ times a perfect square $\xi_5^2\le B^{12}$, detected by a quadratic-residue filter modulo $64$ followed by an exact
integer square root. Each solution is then tested for $I_6\ne0$ and for being canonical (\cref{lem:ratrep}), and
classified. Primitivity only has to be tested at primes $p\le\min_i|\xi_i|^{g/w_i}$, the minimum over the nonzero
coordinates, since $p^{w_i/g}\mid\xi_i\ne0$ forces this; for the points of trivial support gcd $g=1$ this means
$p\le B$, and in general $p\le B^g$.

The work is $\approx16B^{14}$ evaluations of a cubic, $1.1\cdot10^{13}$ for $B=7$. The box is split into slices of values
of $\xi_0$, which are processed independently; on a laptop with $22$ logical processors, running $21$ slices in
parallel, the census takes one hour and fifty minutes. The search finds $215\,382\,151$ tuples, of which $388\,307$ have
$I_6=0$ and $733\,556$ are not canonical, leaving the \Nseven{} points of the census.
This is to be compared with the genus-two census \cite{2026-stats}, where every weighted-primitive tuple is a point
and the cost lies in deciding the obstruction, which requires factoring $J_{30}$. Here the situation is reversed:
finding the points is the expensive step, a search for squares among the values of a cubic, while the
obstruction costs almost nothing, since by \cref{thm:criterion} it is a Hilbert symbol of two numbers, $\xi_0$ of size
$B^2$ and $3\mm$ of size $7B^6$, which are factored by trial division.

The enumeration and the classification were written twice, in C and in Python, and agree on all $142\,344$
points of height at most $3$. Independent runs to heights $5$ and $6$, on a different platform, agree with the census
on their range: the same $12\,438\,121$ canonical representatives of height at most $5$ (compared by a checksum of the
point set), and at every bound $B\le6$ of \cref{tab:census} the same numbers of points, of fine points, of points with
automorphisms, of pseudo-real points and of points of $\Lo2$. All
algebraic identities used (\cref{lem:gram}, \cref{eq:I6}, the strata) were checked symbolically, and the criterion
itself was checked as described in \cref{rem:checks}: every one of the $4\,785\,584$ classes of maps over $\Q$ of
naive height at most $4$ is fine, and every one of $14\,225$ fine points of the census sampled in \cref{sec:models} was
turned into an explicit map over $\Q$ with the right invariants. The product formula is a further check: every
obstruction set of the \Nseven{} points of height at most $7$ has even cardinality (the classification program stops
at the first odd one).

\begin{table}[ht]
\centering
\caption{The census of $\M_3(\Q)$ by weighted moduli height. $\delta=\Nf/N$; $\delta_{\rm ns}$ is the fine proportion among
the points of $U$ not on a split sheet ($-2\xi_0\notin\Q^{\times2}\cup\{0\}$); $N_{\rm sym}$ counts points with nontrivial
automorphism group.}
\label{tab:census}
\begin{tabular}{rrrrrrr}
\toprule
$B$ & $N(B)$ & $\Nf(B)$ & $\delta$ & $\delta_{\rm ns}$ & $N_{\rm sym}$ & $N/(B^8\log B)$\\
\midrule
$1$ & $9$ & $9$ & $1.000$ & --- & $5$ & --- \\
$1.5$ & $248$ & $232$ & $0.935$ & --- & $30$ & $23.9$ \\
$2$ & $3\,539$ & $3\,113$ & $0.880$ & $0.768$ & $182$ & $19.9$ \\
$2.5$ & $29\,152$ & $23\,434$ & $0.804$ & $0.713$ & $618$ & $20.9$ \\
$3$ & $142\,344$ & $111\,414$ & $0.783$ & $0.677$ & $1\,692$ & $19.7$ \\
$3.5$ & $529\,306$ & $389\,017$ & $0.735$ & $0.647$ & $3\,790$ & $18.8$ \\
$4$ & $1\,567\,962$ & $1\,119\,913$ & $0.714$ & $0.635$ & $7\,983$ & $17.3$ \\
$4.25$ & $3\,207\,191$ & $2\,431\,162$ & $0.758$ & $0.620$ & $11\,355$ & $20.8$ \\
$4.5$ & $4\,872\,512$ & $3\,657\,268$ & $0.751$ & $0.612$ & $14\,600$ & $19.3$ \\
$5$ & $12\,438\,121$ & $9\,215\,604$ & $0.741$ & $0.628$ & $27\,735$ & $19.8$ \\
$5.5$ & $28\,316\,999$ & $20\,330\,721$ & $0.718$ & $0.617$ & $47\,731$ & $19.8$ \\
$6$ & $60\,661\,421$ & $42\,804\,136$ & $0.706$ & $0.605$ & $78\,762$ & $20.2$ \\
$6.5$ & $115\,643\,256$ & $79\,697\,105$ & $0.689$ & $0.596$ & $123\,172$ & $19.4$ \\
$7$ & $214\,260\,288$ & $145\,622\,093$ & $0.680$ & $0.594$ & $186\,530$ & $19.1$ \\
\bottomrule
\end{tabular}
\end{table}

\cref{tab:census} gives the counts. The number of points grows like $B^8\log B$ in this range, as predicted in
\cref{sec:counting}: $N(B)/(B^8\log B)$ stays between $17$ and $21$ for $2\le B\le7$; the number of points with nontrivial automorphisms is a negligible fraction ($0.22\%$ at
$B=5$ and $0.09\%$ at $B=7$), as in genus two \cite{2026-stats}, though not as small, since the automorphism loci of $\M_3$ have dimension
two. The fine proportion falls, from $1$ at height $1$ to $0.71$ at height $4$, rises to $0.76$ at height
$4.25$, and falls again, to $0.680$ at height $7$; the explanation of the jump is the subject of the next paragraph.

The symbol $(-\mm,-2\xi_0)$ is trivial whenever $-2\xi_0$ is a square, and for a canonical representative this
happens on the sheets $\xi_0=0,-2,-8,-18,-32,\dots$, on which the divergence form $J_\phi$ has rational roots. These
sheets are heavily overrepresented among points of small height: at $B=7$ the five sheets $\xi_0\in\{0,-2,-8,-18,-32\}$
contain $45\,121\,110$ of the $214\,260\,288$ points ($21.1\%$), and the sheet $\xi_0=-18$ alone $24\,185\,625$, thirteen times the
average non-split sheet. Two effects combine. On the sheet $\xi_0=0$ every point of $U$ is fine by
\cref{thm:criterion}(2), and on a split sheet with $\xi_0\ne0$ the identity \cref{eq:sylv2} becomes
$(h-s\xi_5)(h+s\xi_5)=\mm\,n$ with $s^2=-2\xi_0$, so the double cover acquires the rational threefold
$\{\mm=0\}$, along which $F=h^2/s^2$ is a square identically; of the $8\,609\,849$ points of the census with $\mm=0$,
$8\,592\,638$ lie on split sheets.
And the sheet $\xi_0=-18$ is favoured by the local densities at $2$ and $3$ of the integrality of $F$, whose
denominators divide $216$: all multiples of $6$ are overrepresented, by factors $1.8$ to $3.5$ against the average
non-split sheet. The sheet $\xi_0=-18$ enters the census at height $\sqrt{18}=4.24$, which is the jump in
\cref{tab:census}; the next one, $\xi_0=-32$, enters at $\sqrt{32}=5.66$, where $\delta$ rises again slightly, from
$0.7180$ at $B=5.5$ to $0.7185$ at $B=5.75$. Since the number of split
sheets in the box is $\asymp B$ against $\asymp B^2$ sheets in all, their share should tend to zero; in the census it
falls from $30.2\%$ at $B=5$ to $25.4\%$ at $B=6$ and $21.1\%$ at $B=7$. The fine
proportion away from them, $\delta_{\rm ns}$, is the quantity that measures the obstruction; it decreases,
from $0.77$ at $B=2$ to $0.63$ at $B=5$ and $0.594$ at $B=7$.

\section{Statistics of the obstruction}
\label{sec:statistics}

We restrict to the $214\,073\,758$ points of $U$ of height at most $7$, of which $68\,625\,577$ are coarse.

\begin{figure}[ht]
\centering
\includegraphics[width=0.5\textwidth]{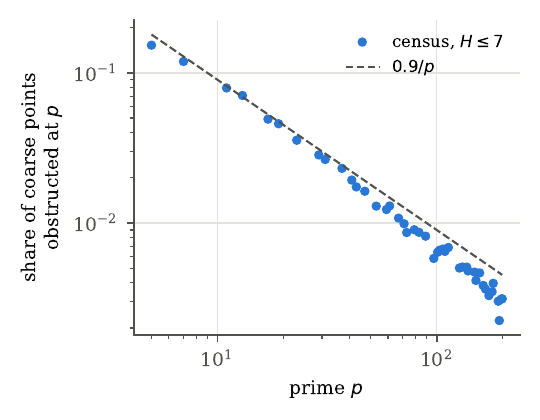}
\caption{Share of the coarse points of $U$ with $H\le7$ obstructed at the prime $p$, with the line $0.9/p$.}
\label{fig:places}
\end{figure}

\subsection{Obstructed places}
Every obstruction set has two, four or six elements ($67\,859\,512$, $765\,840$ and $225$ points). \cref{tab:places}
gives the proportion of coarse points obstructed at the small places and \cref{tab:patterns} the shape of the obstruction
sets. The real place obstructs $3.8\%$ of the coarse points, exactly the $2\,605\,278$ points at which $\xi_0>0$ and
$\mm>0$ (the pseudo-real points); the prime $2$ obstructs $76\%$; and the primes $p\ge5$ obstruct with a frequency
close to $0.9/p$ for small $p$, and somewhat below it for $p\gtrsim50$, where the box is small compared with $p^2$
(\cref{fig:places}). The shape of $S_P$ is the most instructive statistic: $66\%$ of the coarse points are
obstructed at $2$ and at a single prime $p\ge5$, $5\%$ at $3$ and a single $p\ge5$, $1.5\%$ at $\infty$ and a single
$p\ge5$; $8.7\%$ are obstructed only at places in $\{\infty,2,3\}$, and $17\%$ at two primes $p,q\ge5$ and nowhere
else. A point with one obstructing large prime must, by the product formula, be obstructed at a second place, and
since a point rarely reduces into $\Lo2$ at two primes, that second place is almost always $\infty$, $2$ or $3$. In all,
$82.9\%$ of the coarse points are obstructed at one of the places $\infty$, $2$, $3$: the $2$-adic invariant, a locally
constant function of $P$ in the $2$-adic topology, detects obstructions at primes that vary from point to point and are
individually invisible.

\begin{table}[ht]
\centering
\caption{Proportion of the $68\,625\,577$ coarse points of $U$ with $H\le7$ obstructed at a given place.}
\label{tab:places}
\begin{tabular}{lccccccccc}
\toprule
place & $\infty$ & $2$ & $3$ & $5$ & $7$ & $11$ & $13$ & $17$ & $19$\\
\midrule
\% & $3.80$ & $75.91$ & $12.44$ & $15.30$ & $11.91$ & $7.94$ & $7.07$ & $4.92$ & $4.59$\\
\bottomrule
\end{tabular}
\end{table}

\begin{table}[ht]
\centering
\caption{Obstruction sets of the coarse points of $U$ with $H\le7$: the small places in $S_P$ and the number of
primes $p\ge5$ in $S_P$.}
\label{tab:patterns}
\begin{tabular}{lrr}
\toprule
$S_P\cap\{\infty,2,3\}$, $\#(S_P\cap\{p\ge5\})$ & points & \% \\
\midrule
$\{2\}$, $1$ & $45\,534\,304$ & $66.4$\\
$\emptyset$, $2$ & $11\,749\,105$ & $17.1$\\
$\{2,3\}$, $0$ & $4\,518\,513$ & $6.6$\\
$\{3\}$, $1$ & $3\,563\,668$ & $5.2$\\
$\{\infty,2\}$, $0$ & $1\,313\,031$ & $1.9$\\
$\{\infty\}$, $1$ & $1\,052\,278$ & $1.5$\\
other & $894\,678$ & $1.3$\\
\bottomrule
\end{tabular}
\end{table}

\subsection{The mechanism prime by prime}
By \cref{thm:local}, a prime $p\ge5$ can obstruct only if $G(P)$ has rank at most one modulo $p$, that is, only if
$P\bmod p\in Z$ or $p$ divides every coordinate of the canonical representative. In the census all
$75\,546\,209$ pairs $(P,p)$ with $p\in S_P$, $p\ge5$, satisfy this ($12\,218$ of them through the second
alternative), and both statements of \cref{thm:local}(3) hold without exception, on $324\,819\,419$ and
$173\,894\,528$ pairs. Let $\omega(P)$ be the number of primes $p\ge5$ at which $G(P)$ has rank at most one modulo
$p$; for all but at most $289\,361$ of the points of $U$ in the census this is the number of primes with $P\bmod p\in Z$.
\cref{tab:omega} shows the fine proportion as a
function of $\omega$: it is multiplied by a factor $0.60$--$0.68$ for each of the first three additional primes, as
\cref{eq:model} predicts with $q_p\approx0.3$--$0.4$, and by $0.85$ from $\omega=3$ to $\omega=4$. The mean of $\omega$ grows with the height, from $0.23$ for $H\le2$ to $0.48$, $0.67$,
$0.77$, $0.86$ and $0.94$ on the shells $(2,3]$, $(3,4]$, $(4,5]$, $(5,6]$, $(6,7]$.

\begin{table}[ht]
\centering
\caption{Fine proportion as a function of $\omega(P)=\#\{p\ge5: \operatorname{rank}(G(P)\bmod p)\le1\}$, points of $U$ with $H\le7$. (There are
also $60$ points with $\omega=5$, one of them fine.)}
\label{tab:omega}
\begin{tabular}{lrrrrr}
\toprule
$\omega$ & $0$ & $1$ & $2$ & $3$ & $4$\\
\midrule
points & $71\,756\,177$ & $95\,516\,814$ & $41\,552\,044$ & $5\,125\,243$ & $123\,420$\\
fine proportion & $0.957$ & $0.626$ & $0.376$ & $0.254$ & $0.215$\\
\bottomrule
\end{tabular}
\end{table}

In the census, and in the $p$-adic computation below, we measure the event $P\bmod p\in Z$ by the rank condition
$\operatorname{rank}(G(P)\bmod p)\le1$, which differs from it only at the primes dividing every coordinate. The
probability $q_p$ that a prime satisfying the rank condition obstructs, and the frequency $\tau_p/p$ of the rank
condition, can be predicted from $p$-adic densities: a point $x\in\Z_p^5$ is weighted by the local number of
square roots of $F(x)$, $(1+\chi(u))p^{v/2}$ if $F(x)=p^vu$ with $v$ even and $0$ otherwise, and one integrates the
indicator of the rank condition $\operatorname{rank}(G(x)\bmod p)\le1$ and of $\inv_p\ne0$, which is computed from the
Hilbert symbol. \cref{tab:qp} compares a Monte-Carlo
evaluation of these integrals ($4\cdot10^6$ samples modulo $p^8$, or $p^6$ for $p\ge17$) with the census. The prediction
agrees with \cref{sec:counting}: $\tau_p=1$ within the sampling error, so that the codimension-two locus $Z$ is met with
the frequency of a divisor, and $q_p$ increases towards $\tfrac12$, roughly as $\tfrac12-0.8/p$; and in the sample no
point outside the tube defined by the rank condition is obstructed. The census agrees qualitatively; its values of $\tau_p$ are smaller
($0.65$--$0.78$ on the shell $6<H\le7$), but for each of these primes they increase from shell to shell (for $p=5$:
$0.716$, $0.745$, $0.755$, $0.780$ on the shells $H\le4$, $4<H\le5$, $5<H\le6$, $6<H\le7$), which is what one expects from a box
whose sides $B^2=49$ in $\xi_0,\xi_1$ are comparable to $p^2$; its values of $q_p$ are smaller on all points, because on
the split sheets $q_p=0$, and larger on the points off the split sheets.

\begin{table}[ht]
\centering
\caption{The probability that $G(P)$ has rank at most one modulo $p$ (as $\tau_p=p\cdot\Pr$) and the probability $q_p$ that $p$
then obstructs: $p$-adic prediction and census (shell $6<H\le7$; all points of $U$, and points off the split sheets).}
\label{tab:qp}
\begin{tabular}{rcccccc}
\toprule
 & \multicolumn{2}{c}{prediction} & \multicolumn{2}{c}{census, all} & \multicolumn{2}{c}{census, non-split}\\
$p$ & $\tau_p$ & $q_p$ & $\tau_p$ & $q_p$ & $\tau_p$ & $q_p$\\
\midrule
$5$ & $1.016$ & $0.335$ & $0.780$ & $0.312$ & $0.784$ & $0.385$ \\
$7$ & $1.007$ & $0.373$ & $0.750$ & $0.365$ & $0.756$ & $0.449$ \\
$11$ & $0.988$ & $0.418$ & $0.720$ & $0.402$ & $0.731$ & $0.492$ \\
$13$ & $0.967$ & $0.420$ & $0.702$ & $0.439$ & $0.714$ & $0.535$ \\
$17$ & $0.986$ & $0.443$ & $0.684$ & $0.411$ & $0.697$ & $0.500$ \\
$19$ & $0.983$ & $0.464$ & $0.677$ & $0.424$ & $0.700$ & $0.509$ \\
$23$ & $0.951$ & $0.461$ & $0.653$ & $0.414$ & $0.676$ & $0.497$ \\
$29$ & $1.008$ & $0.466$ & $0.645$ & $0.416$ & $0.662$ & $0.503$ \\
$31$ & $1.013$ & $0.469$ & $0.646$ & $0.419$ & $0.662$ & $0.507$ \\
\bottomrule
\end{tabular}
\end{table}

\subsection{The fine proportion and the height}
\cref{fig:height} shows the counting functions and the fine proportions. The overall proportion is dominated at
small height by the split sheets; away from them the proportion $\delta_{\rm ns}$ decreases from $0.77$ at $B=2$ to
$0.63$ at $B=5$ and $0.594$ at $B=7$, while the mean of $\omega$ increases. The range is far too short to test the exponent $\tfrac12$ of
\cref{conj:density}. At $B=5,6,7$ the product $\delta(B)(\log B)^{1/2}$ equals $0.940$, $0.945$, $0.948$, nearly
constant, and $\delta(B)\log B$ equals $1.19$, $1.26$, $1.32$, so that in this range the fine proportion decays more
slowly than $(\log B)^{-1}$; but $\delta_{\rm ns}(B)(\log B)^{1/2}$ still increases ($0.80$, $0.81$, $0.83$), and three
values of $B$ cannot determine an exponent. What the census shows, in its range, is the mechanism through
which the decay should happen, and the values $\tau_p\approx1$ and $q_p\to\tfrac12$ of its quantitative ingredients.

\begin{figure}[ht]
\centering
\includegraphics[width=\textwidth]{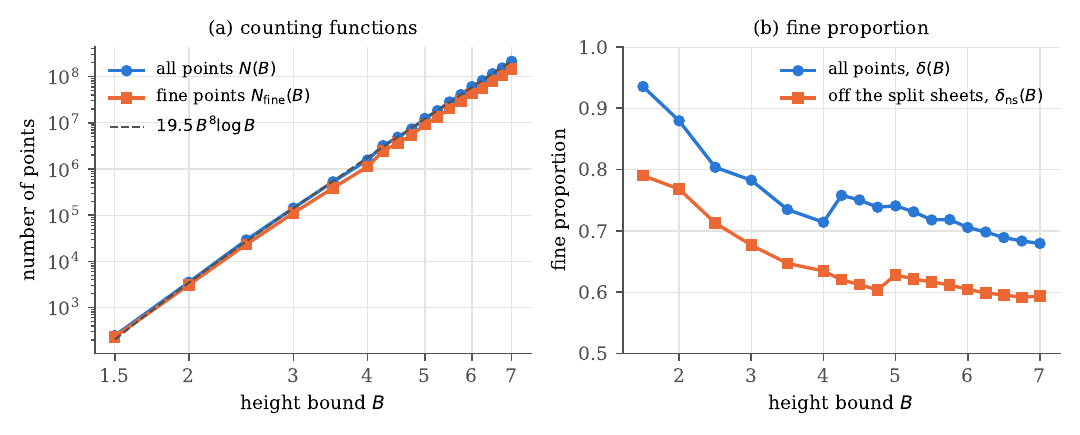}
\caption{(a) The counting functions of the census against $19.5\,B^8\log B$. (b) The fine proportion $\delta(B)$ and the fine
proportion $\delta_{\rm ns}(B)$ among the points of $U$ off the split sheets.}
\label{fig:height}
\end{figure}

\subsection{Dependence on the height}
\cref{tab:shells} shows the census shell by shell. The share of the coarse points obstructed at $2$ falls steadily, from
$96\%$ for $H\le2$ to $75\%$ on the shell $6<H\le7$, and the mean of $\omega$ grows from $0.23$ to $0.94$; on the last two
shells it grows by $0.09$ and $0.07$, comparable to the increments $0.11$ and $0.08$ of $\log\log B$, as the heuristic of
\cref{sec:counting} predicts, without proving anything about its growth. The fine proportion at fixed $\omega$ is nearly
constant for $\omega=0$. For $\omega\ge1$ it depends on the composition of the shell: it is largest on the shell
$4<H\le5$, in which the split sheet $\xi_0=-18$ enters (\cref{sec:census}), all of whose points are fine, and it
decreases over the last three shells ($0.663$, $0.631$, $0.623$ for $\omega=1$, and $0.311$, $0.273$, $0.249$ for
$\omega=3$). The pattern $\{p,q\}$ of two large primes, which reciprocity does not transfer to a small place, makes up
$17.8\%$ of the coarse points on the shell $6<H\le7$, and obstruction sets with six places occur there ($183$ points).

\begin{table}[ht]
\centering
\caption{The census of $U$ by height shell: number of points, fine proportion, mean of $\omega$, share of the coarse points
obstructed at $2$, and fine proportion as a function of $\omega$ (omitted for fewer than $100$ points).}
\label{tab:shells}
\begin{tabular}{lrccccccc}
\toprule
 & & & & & \multicolumn{4}{c}{fine proportion at $\omega=$}\\
shell & points & fine & mean $\omega$ & obstructed at $2$ & $0$ & $1$ & $2$ & $3$\\
\midrule
$H\le2$ & $3\,357$ & $0.876$ & $0.23$ & $96.1\%$ & $0.968$ & $0.569$ & --- & --- \\
$2<H\le3$ & $137\,295$ & $0.779$ & $0.48$ & $89.0\%$ & $0.954$ & $0.582$ & $0.287$ & --- \\
$3<H\le4$ & $1\,419\,327$ & $0.706$ & $0.67$ & $84.1\%$ & $0.946$ & $0.574$ & $0.295$ & $0.144$ \\
$4<H\le5$ & $10\,850\,407$ & $0.744$ & $0.77$ & $79.5\%$ & $0.961$ & $0.663$ & $0.430$ & $0.311$ \\
$5<H\le6$ & $48\,172\,273$ & $0.696$ & $0.86$ & $77.4\%$ & $0.958$ & $0.631$ & $0.385$ & $0.273$ \\
$6<H\le7$ & $153\,491\,099$ & $0.669$ & $0.94$ & $75.2\%$ & $0.956$ & $0.623$ & $0.371$ & $0.249$ \\
\bottomrule
\end{tabular}
\end{table}

\begin{figure}[ht]
\centering
\includegraphics[width=\textwidth]{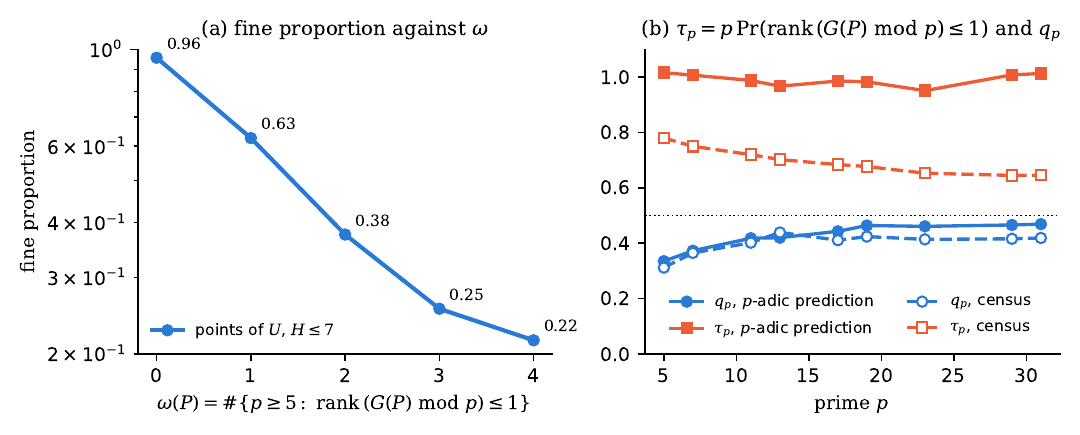}
\caption{(a) Fine proportion as a function of the number $\omega(P)$ of primes $p\ge5$ at which $G(P)$ has rank at most
one modulo $p$ (logarithmic scale), for the points of $U$ with $H\le7$. (b) The frequency $\tau_p$ of the rank condition and the probability $q_p$ that such a prime obstructs:
$p$-adic prediction (filled) and census shell $6<H\le7$ (hollow).}
\label{fig:mechanism}
\end{figure}

\subsection{Minimal fields of definition}
By \cref{cor:consequences}(4), the quadratic fields of definition of a coarse point are determined by $S_P$. Taking for
each coarse point of $U$ with $H\le7$ the squarefree $d$ of smallest absolute value with $\Q(\sqrt d)$ splitting
$(-\mm,-2\xi_0)$ (trying $d>0$ before $-d$), we find that $50.1\%$ of the coarse points have a model over $\Q(i)$,
$30.4\%$ over $\Q(\sqrt2)$, $10.2\%$ over $\Q(\sqrt3)$, $4.3\%$ over $\Q(\sqrt{-2})$ and $2.7\%$ over $\Q(\sqrt5)$;
$84.8\%$ need $|d|\le2$, $99.85\%$ need $|d|\le10$, and the largest value is $|d|=67$. The dominance of $\Q(i)$ reflects the dominance of the obstruction sets $\{2,p\}$: $-1$ is a non-square
at $2$ and at the $p\equiv3\pmod4$. As in genus two \cite{2026-stats}, coarse points are barely coarse.

\begin{figure}[ht]
\centering
\includegraphics[width=\textwidth]{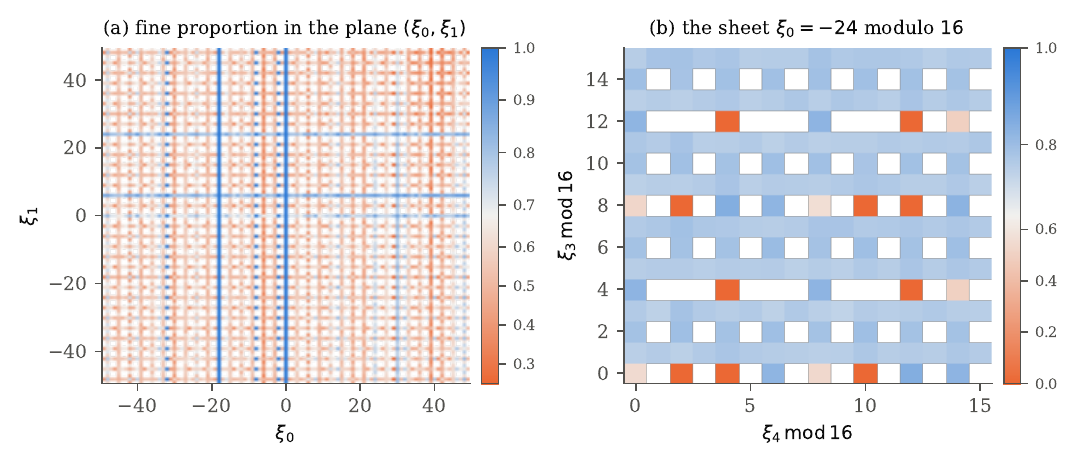}
\caption{(a) Fine proportion of the points of $U$ with $H\le7$ in each cell $(\xi_0,\xi_1)$ with at least $50$ points, on a
diverging scale centred at the overall proportion; the split sheets $\xi_0=0,-2,-8,-18,-32$ are the dark columns. (b) The sheet $\xi_0=-24$, aggregated by
$(\xi_3\bmod16,\ \xi_4\bmod16)$: the $2$-adic imprint of the obstruction (white: no points).}
\label{fig:slice}
\end{figure}

\subsection{A picture of the moduli space}
\cref{fig:slice} shows the fine proportion over two coordinate planes of the census. In the plane $(\xi_0,\xi_1)$ the
split sheets $\xi_0\in\{0,-2,-8,-18,-32\}$ are entirely fine and the region $\xi_0>0$ carries the pseudo-real points; within
a sheet, the fine proportion varies with congruence conditions on $\xi_1,\xi_3,\xi_4$ modulo powers of $2$ and $3$, which
are the $2$- and $3$-adic invariants of $(-\mm,-2\xi_0)$. The odd primes, which in the model \cref{eq:model} carry the
decay, leave no visible structure: their contribution depends on the reduction of the point modulo a prime that
varies from point to point.

\section{Minimal models}
\label{sec:models}

For a fine point $P$ the natural next object is a rational map over $\Q$ in the class $P$ of smallest naive height
$h(\phi)=\max|c_i|$ (with $c_0,\dots,c_7$ coprime integers),
\[
h_{\min}(P)=\min\{h(\phi):\ \phi\in\Rat_3(\Q),\ [\phi]=P\}.
\]
The database of \cite{2024-4-arxiv} contains every map with $h\le4$ and therefore determines $h_{\min}(P)$ whenever it is
at most $4$; it does so for $4\,785\,584$ classes. The moduli heights of these classes range from $1$ to $148$, with
median $45$ for $h_{\min}=4$, $34$ for $h_{\min}=3$, $24$ for $h_{\min}=2$ and $15$ for $h_{\min}=1$: a map of naive height $T$ has
invariants of size $T^{w_i}$ and moduli height of order $T$, with a large constant. Conversely, only $54\,359$ of the
$145\,622\,093$ fine points of the census, $0.04\%$, have $h_{\min}\le4$. To go beyond the database one needs to construct
models.

\begin{alg}[Small model of a fine point]\label{alg:small}
Input: a fine point $P\in U(\Q)$ with $\xi_0\mm\ne0$.
\begin{enumerate}
\item Construct a map $\phi_0$ over $\Q$ with invariants $\xi(P)$ by \cref{prop:explicit}; the conic is solved by
\texttt{qfsolve} \cite{pari}.
\item Minimise the resultant: at each prime $p$ dividing $\lambda^6I_6(P)$, where $\lambda$ is the scaling between the
invariants of $\phi_0$ and the canonical representative, move greedily in the Bruhat--Tits tree, conjugating by
$\left(\begin{smallmatrix}p&a\\0&1\end{smallmatrix}\right)$ and $\left(\begin{smallmatrix}1&0\\0&p\end{smallmatrix}\right)$, as long as
$|\Res|$ decreases. Since $\operatorname{ord}_p\Res$ is convex on the tree \cite{rumely-2015} and changes by multiples of
$6$, this reaches a minimal model at every prime \cite{bruin-molnar-2012}; only primes with $p^6\mid\Res$ need to be
visited.
\item Collect models of equal resultant: the minimal resultant locus can be a segment \cite{rumely-2015}, and we
take the models of equal resultant reached by at most three successive moves in the tree at the primes $p$ with
$p^2\mid\Res$. (A primitive matrix of determinant $p$ multiplies the resultant by $p^{12-6e}$, where $p^e$ is the
content of the conjugate, so equality needs $e=2$.) This is a truncation: the segment can be longer than three, as
for $\frac{z^3-p^{16}}{z}$, whose nine models $\frac{p^kz^3-p^{16-2k}}{z}$, $0\le k\le8$, all have resultant $-p^{16}$.
\item Reduce at infinity: for each of them, conjugate by the element of $\SL_2(\Z)$ that moves a covariant point into the
standard fundamental domain (the roots in the upper half plane of the definite covariant quadratics $J$, $(I,J)_2$,
$u_3$ and their sums, and of the fixed-point quartic $I$, computed with $60$ significant digits when the coefficients exceed $10^9$), descend greedily under
$z\mapsto z\pm1$, $-1/z$, $-z$, $z/(z\pm1)$, repeat until no covariant point moves, and search all words of length at most
$4$ in these generators. When all covariant quadratics are indefinite and all fixed points are real, use instead, when $I(z,1)$ has degree four, Julia's covariant point of $I$, the minimiser of $\sum_j t_j|z-\alpha_j|^2/\operatorname{Im}z$ over the roots $\alpha_j$ of $I(z,1)$ with $t_j=|\partial_zI(\alpha_j,1)|^{-1}$ \cite{hutz-stoll-2019}, and, when it has degree three, the analogous point formed from its three finite roots, as a heuristic; in both cases provided $I(z,1)$ has no repeated root (a fixed point of multiplier $1$) and the roots are numerically resolved, and otherwise keep the current model. We apply both.
\item Local search: from the best model so far, search breadth-first, to depth $3$, over the moves $z\mapsto z\pm1$, $-1/z$,
$-z$ and the conjugations by $\left(\begin{smallmatrix}p&a\\0&1\end{smallmatrix}\right)$ and
$\left(\begin{smallmatrix}1&0\\0&p\end{smallmatrix}\right)$ for $p=2,3$, and repeat while the naive height decreases.
\item Return the model of smallest naive height found.
\end{enumerate}
\end{alg}

Steps (3)--(5) are heuristic. Step (5) leaves the minimal resultant locus on purpose: a model of smallest naive height need not have minimal resultant.

\begin{exa}\label{ex:nonminimal}
The map $\phi(x)=\dfrac{x^3+2x^2-x+6}{x^2+2x+3}$ has $|\Res\phi|=132$, which is minimal at every prime since $p^6\nmid132$,
and its class is the fine point $[8:0:0:20:-2:2]$. Its conjugate by $\operatorname{diag}(2,1)$,
\[
\frac{\phi(2x)}{2}=\frac{4x^3+4x^2-x+3}{4x^2+4x+3},
\]
has resultant $2^6\cdot132$ but naive height $4$, which is $h_{\min}$ of the class by the database of \cite{2024-4-arxiv}; steps (1)--(4), which stay on the
minimal resultant locus, return a model of naive height $6$.
\end{exa}

The reduction step is a heuristic version of the reduction theory of \cite{hutz-stoll-2019}, and the
output is an upper bound $h_{\rm alg}(P)\ge h_{\min}(P)$. We tested it on the database of \cite{2024-4-arxiv}: for
random classes of the generic stratum with $\xi_0\mm\ne0$ and $h_{\min}=1,2,3,4$ (about $500$ of each, which for $h_{\min}=1$ is all of them), starting from the moduli point alone, the algorithm returns a model of
naive height exactly $h_{\min}$ in $100.0\%$ of the cases (\cref{tab:valid}). Steps (3)--(5) matter: with step (2) and a greedy reduction alone the success rate is $93\%$ for $h_{\min}=3,4$;
computing the covariant points in double precision fails for some models with coefficients of $15$ digits or more, which
the raw output of step (1) often has; without step (5) the algorithm misses $h_{\min}$ for classes such as the one of
\cref{ex:nonminimal}; and the class $[14:6:6:-64:-64:0]$, with minimal resultant $2116=46^2$, has a model of
naive height $4$, but it differs from the model of naive height $6$ found when step (3) is restricted to $p\le7$ by a
conjugation of determinant $46=2\cdot23$ that preserves the resultant. On a second test set, $5\,433$ census points with
$h_{\min}\le4$ and $\xi_0\mm\ne0$ in the generic stratum (drawn separately from $H\le5$ and from $5<H\le7$: in each range
all such points with $h_{\min}\le2$ and random samples of $1\,100$ points with $h_{\min}=3,4$, less those with
$\xi_0\mm=0$), the algorithm returns a model of naive height $h_{\min}$ in all cases but one. The exception is the class
$[-48:2:0:48:-48:-256]$, of height $48^{1/2}$: the algorithm returns $\frac{2x^3-3x^2+2x}{3x-2}$, of naive height $3$,
while the database contains $\frac{x^3-x^2+2x}{2x^2-x-2}$, of naive height $2$; the two have the same resultant $-64$, and
the second is the conjugate of the first by $x\mapsto\frac{2x}{x+2}$, of determinant $4$. Steps (3)--(5) are
heuristic, and $h_{\rm alg}$ can exceed $h_{\min}$.

\begin{table}[ht]
\centering
\caption{Validation of \cref{alg:small} on classes of the database of \cite{2024-4-arxiv} with known $h_{\min}\le4$.}
\label{tab:valid}
\begin{tabular}{lcccc}
\toprule
$h_{\min}$ & $1$ & $2$ & $3$ & $4$\\
\midrule
classes tested & $482$ & $500$ & $498$ & $496$ \\
exact $h_{\min}$ found & $100.0\%$ & $100.0\%$ & $100.0\%$ & $100.0\%$ \\
largest height returned & $1$ & $2$ & $3$ & $4$ \\
\bottomrule
\end{tabular}
\end{table}

\subsection{The cost of descent}
We ran \cref{alg:small} on a stratified random sample of $14\,225$ fine points of the census ($2\,500$ per
height shell, all $1\,725$ eligible points with $H\le2$), taking a median of $0.51$ seconds per point; every model was
checked to have the right invariants. Of the sampled points, $127$ are classes of the database of \cite{2024-4-arxiv}; for all of them the model found has naive height exactly $h_{\min}$. Compared with the variant of \cref{alg:small} in which step (3) is restricted to $p\le7$ and distance two and step (5) is omitted, the algorithm finds a smaller model for $25.4\%$ of the sample, by a factor of $1.45$ in the median. \cref{fig:models} and \cref{tab:hmin} show the result. The naive height $h_{\rm alg}$ of the model found is only
loosely tied to the moduli height: its median grows from $18$ for $H\le2$ to $72.5$ for $4<H\le5$, $103$ for $5<H\le6$
and $165.5$ for $6<H\le7$, while its upper tail is very long, and the correlation between $\log h_{\rm alg}$ and $\log H$
is only $0.40$. The size of a small model is
governed by arithmetic rather than by the height: the scaling $\lambda$ between the invariants of a model and the canonical
representative must absorb the weighted content of the invariants, and the minimal resultant is $\lambda^6I_6(P)$ for the
best $\lambda$.

\begin{table}[ht]
\centering
\caption{Naive height of the model found by \cref{alg:small} for fine points of the census, by height shell.}
\label{tab:hmin}
\begin{tabular}{lrrrrr}
\toprule
shell & points & median & $75\%$ & $90\%$ & share $\le10$\\
\midrule
$H\le2$ & $1725$ & $18$ & $45$ & $80$ & $31.1\%$ \\
$2<H\le3$ & $2500$ & $31$ & $88$ & $550$ & $14.9\%$ \\
$3<H\le4$ & $2500$ & $54$ & $239$ & $2310$ & $7.3\%$ \\
$4<H\le5$ & $2500$ & $72.5$ & $489$ & $7022$ & $3.4\%$ \\
$5<H\le6$ & $2500$ & $103$ & $1126$ & $19010$ & $1.6\%$ \\
$6<H\le7$ & $2500$ & $165.5$ & $2723$ & $43998$ & $0.6\%$ \\
\bottomrule
\end{tabular}
\end{table}

\begin{figure}[ht]
\centering
\includegraphics[width=0.6\textwidth]{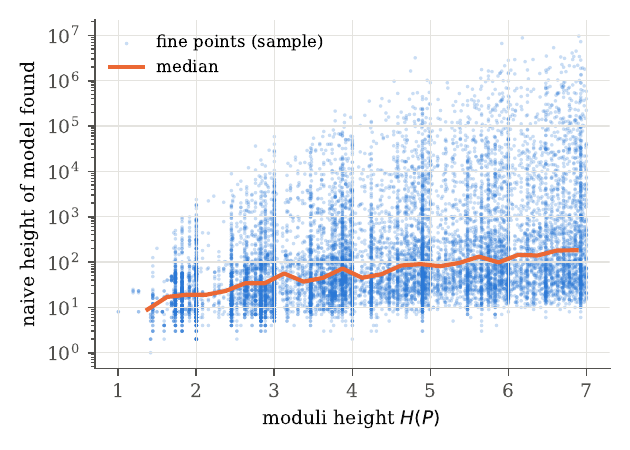}
\caption{Naive height of the model produced by \cref{alg:small} against the moduli height, for the fine points of the census
in the sample of \cref{tab:hmin}, with the running median.}
\label{fig:models}
\end{figure}

\section{Integer points in orbits}
\label{sec:orbits}

The descent and minimisation of \cref{sec:models} make the census an ordering of the rational cubic maps themselves, and
not only of their moduli points. In this section we use it to formulate the cubic case of the conjecture of \cite[Conj.~3.47]{silverman-2007-ads} on
the number of integers in a wandering orbit, and we record what follows from known results on the generic stratum. No
orbit statistics are asserted here.

\subsection{The arithmetic normalisation}
Write $\phi=f/g\in\Q(z)$ with $f,g$ coprime and a primitive integral coefficient vector, put $d=\max(\deg f,\deg g)$,
$F_\phi(X,Y)=Y^df(X/Y)$, $G_\phi(X,Y)=Y^dg(X/Y)$, and $R(\phi)=|\Res_d(F_\phi,G_\phi)|$, the resultant of two binary forms of
degree $d$. We call $\phi$ \emph{resultant-minimal} if $R(\phi)=\min_{A\in\PGL_2(\Q)}R(\phi^A)$. This is a minimum within
a rational conjugacy class, distinct from the minimisation of the naive height in \cref{sec:models}: by
\cref{ex:nonminimal}, a model of smallest naive height need not be resultant-minimal. Minimality under the
affine transformations $z\mapsto az+b$, $a\in\Q^\times$, $b\in\Q$, and under all of $\PGL_2(\Q)$ agree
\cite[Prop.~2.10]{bruin-molnar-2012}.

For $\alpha\in\Q$ write $\mathcal O_\phi(\alpha)=\{\phi^n(\alpha):n\ge0\}$, where $\phi^n$ is the $n$-th iterate, and
$N_\Z(\phi,\alpha)=\#(\mathcal O_\phi(\alpha)\cap\Z)$; the point $\alpha$ is \emph{wandering} if its orbit is infinite.
The number $N_\Z(\phi,\alpha)$ is finite when $d\ge2$ and $\phi^2\notin\Q[z]$ \cite[\S3.7]{silverman-2007-ads}, and it is
conjectured that for resultant-minimal $\phi$ it is bounded by a constant depending only on $d$
\cite[Conj.~3.47]{silverman-2007-ads}.
The moduli height orders the maps, but it cannot replace the normalisation.

\begin{exa}[Fixed moduli height without minimality]\label{ex:orbit-scaling}
Let $\phi(z)=\dfrac{z^3+z+1}{z^2}$ and $\psi_D(z)=D\phi(z/D)=\dfrac{z^3+D^2z+D^3}{z^2}$ for integers $D\ge1$. The orbit of
$1$ under $\phi$ is positive and strictly increasing, since $\phi(x)=x+x^{-1}+x^{-2}>x$ for $x>0$, so $1$ is wandering.
Since $\psi_D^n(D)=D\phi^n(1)$, choosing $D$ divisible by the denominators of the first $N$ points of this orbit makes the
orbit of $D$ under $\psi_D$ begin with $N$ integers; for example $\psi_9$ gives $9\mapsto27\mapsto31\mapsto33031/961$.
All the $\psi_D$ are conjugate over $\Q$, so they have the same moduli point, $[8:0:0:16:-2:-2]$, of height $8^{1/2}$,
while $R(\phi)=1$ and $R(\psi_D)=D^6$. The second iterates are not polynomials: $0$ is a pole and $\infty$ is fixed.
The point lies in $U$: the fixed-point quartic of $\phi$ is $Y^3(X+Y)$, so an automorphism fixes $\infty$ and $-1$, hence
also the only finite pole $0$, and is the identity. Thus no bound in terms of the moduli height alone can hold without
resultant minimality, even on $U$.
\end{exa}

\subsection{A height formulation of the cubic conjecture}
For $P\in\M_3(\Q)$ let $\mathcal R(P)$ be the set of resultant-minimal $\phi\in\Rat_3(\Q)$ with $[\phi]=P$. If $P$ has
automorphisms, $\mathcal R(P)$ contains the minimal models of all rational forms of $P$, not only the conjugates of one
of them. Put
\[
\mathcal I_3(P)=\sup\bigl\{N_\Z(\phi,\alpha):\ \phi\in\mathcal R(P),\ \phi^2\notin\Q[z],\ \alpha\in\Q\ \text{wandering for}\ \phi\bigr\}
\]
(with $\sup\varnothing=0$, in particular at a coarse point) and, for $B\ge1$,
\[
M_3(B)=\sup\{\mathcal I_3(P):\ P\in\M_3(\Q),\ H(P)\le B\}.
\]

\begin{conj}\label{conj:integral-orbits}
There is a constant $C_3$ such that $M_3(B)\le C_3$ for every $B\ge1$.
\end{conj}

Since every rational cubic map has a moduli point of finite height and $\mathcal I_3$ admits all minimal models and all
wandering starting points, this is exactly the cubic case of \cite[Conj.~3.47]{silverman-2007-ads}, rewritten through
the height filtration. A bound $M_3(B)\le F(B)$ growing with $B$ would be weaker. Minimal cubic maps with ten integers
in an orbit are given in \cite[\S7]{bruin-molnar-2012}, so $C_3\ge10$. The search there prescribes the first
integers of an orbit and interpolates a map through them; the census gives instead the complete list of moduli points of
bounded height, and descent decides which of them carry rational maps.

\subsection{Rational forms}
A moduli point and a rational conjugacy class are different things, and the difference matters here.

\begin{lem}\label{lem:unique-rational-form}
If $P\in U(\Q)$ is fine, all maps over $\Q$ in the class $P$ are conjugate by elements of $\PGL_2(\Q)$.
\end{lem}
\begin{proof}
Let $\psi=\phi^A$ with $\phi,\psi$ defined over $\Q$ and $A\in\PGL_2(\overline\Q)$. For every
$\tau\in\operatorname{Gal}(\overline\Q/\Q)$ also $\psi=\phi^{A^\tau}$, so $A^\tau A^{-1}$ is an automorphism of $\phi$,
which is trivial since $P\in U$. Hence $A$ is fixed by Galois, and $A\in\PGL_2(\Q)$.
\end{proof}

On the strata with automorphisms a single moduli point can carry infinitely many rational forms, all minimal.

\begin{prop}\label{prop:fixed-height-twists}
For squarefree integers $t\ge1$ the maps $\phi_t(z)=\dfrac{tz^3+z}{tz^2+2}$ have the same moduli point
$[56:0:0:0:-2:2]\in\Lo1$, are resultant-minimal with $R(\phi_t)=2t^3$, are pairwise not conjugate over $\Q$, and have
$\phi_t^2\notin\Q[z]$.
\end{prop}
\begin{proof}
With $A_t(z)=\sqrt t\,z$ one has $\phi_t=\phi_1^{A_t}$, so the moduli point does not depend on $t$; it is the one given,
which lies on $\xi_2=\xi_3=0$. The forms $tX^3+XY^2$ and $tX^2Y+2Y^3$ have coprime integral coefficients and resultant
$2t^3$. Since $t$ is squarefree, $v_p(2t^3)\in\{0,1,3,4\}$ is smaller than $6$ at every prime, and for odd degree $d$ a
model with $v_p(\Res)<2d$ is minimal at $p$ \cite[Lemma~3.1]{bruin-molnar-2012}. Distinct $t$ give distinct minimal
resultants, so the maps are not conjugate over $\Q$. Finally the roots of $tz^2+2$ are poles which $\phi_t$ maps to the
fixed point $\infty$, so they are poles of $\phi_t^2$.
\end{proof}

So the finitely many moduli points of bounded height do not give finitely many rational conjugacy classes on all of
$\M_3$: the maps $\phi_t$ have constant moduli height and unbounded minimal resultant. On the fine points of $U$ they do,
by \cref{lem:unique-rational-form}.

\subsection{A bound on the generic stratum}
Two points of $\P^1(\Q)$ are integral relative to each other if their reductions are distinct at every prime; a rational
number is an integer exactly when it is integral relative to $\infty$, and $\PGL_2(\Z)$ preserves relative integrality.
A point is exceptional for $\psi$ if its backward orbit is finite; $\infty$ is exceptional for $\phi$ exactly when
$\phi^2$ is a polynomial. We use the estimate of \cite[Thm.~1.2]{yap-2026-integral}, specialised
to $K=\Q$, $S=\{\infty\}$ and degree $3$: there is a constant $c_3$ such that for $\psi\in\Rat_3(\Q)$, every
$\gamma\in\P^1(\Q)$ that is not preperiodic and every $\eta\in\P^1(\Q)$ that is not exceptional,
\begin{equation}\label{eq:yap}
\#\{n\ge0:\ \psi^n(\gamma)\ \text{is integral relative to}\ \eta\}\ \le\ c_3+\log_3^+\frac{h(\psi)+1}{\widehat h_\psi(\gamma)},
\end{equation}
where $h(\psi)=\log\max|c_i|$ is the logarithmic height of the primitive coefficient vector, $\widehat h_\psi$ the
canonical height and $\log_3^+x=\max(0,\log x/\log3)$. The bound does not depend on $\eta$.

For a fine $P\in U(\Q)$ and a model $\phi_0$ of $P$ over $\Q$ put
\[
m(P)=\min\{\widehat h_{\phi_0}(\alpha):\ \alpha\in\P^1(\Q)\ \text{wandering for}\ \phi_0\}>0 ;
\]
the minimum exists and is positive by the Northcott property, and it does not depend on $\phi_0$, since a rational
conjugation is a bijection of $\P^1(\Q)$ preserving canonical heights. By \cref{lem:unique-rational-form} the minimal
models of $P$ form one rational conjugacy class, and by \cite[Prop.~6.5]{bruin-molnar-2012} they form finitely many
$\PGL_2(\Z)$-orbits $\mathcal C_1(P),\dots,\mathcal C_{s(P)}(P)$. Put
$r(P)=\max_j\min_{\psi\in\mathcal C_j(P)}h(\psi)$, the largest of the smallest logarithmic heights in these orbits.

\begin{prop}\label{prop:generic-orbit-bound}
For every fine $P\in U(\Q)$,
\[
\mathcal I_3(P)\le c_3+\log_3^+\frac{r(P)+1}{m(P)} .
\]
Consequently $\sup\{\mathcal I_3(P):\ P\in U(\Q)\ \text{fine},\ H(P)\le B\}$ is finite for every $B\ge1$.
\end{prop}
\begin{proof}
Let $\phi\in\mathcal R(P)$ with $\phi^2\notin\Q[z]$, and $\alpha\in\Q$ wandering. Then $\phi$ lies in some
$\mathcal C_j(P)$; let $\psi$ attain the minimum of $h$ on it and write $\phi=\psi^A$ with $A\in\PGL_2(\Z)$. Since $A$
preserves relative integrality, $\phi^n(\alpha)\in\Z$ if and only if $\psi^n(A\alpha)$ is integral relative to
$A(\infty)$, which is not exceptional for $\psi$ because $\infty$ is not exceptional for $\phi$. As $\alpha$ is
wandering, distinct $n$ give distinct points, so $N_\Z(\phi,\alpha)$ is the left side of \cref{eq:yap} for $\psi$,
$\gamma=A\alpha$, $\eta=A(\infty)$, and $h(\psi)\le r(P)$, $\widehat h_\psi(A\alpha)=\widehat h_\phi(\alpha)\ge m(P)$.
The last assertion follows because there are finitely many points of height at most $B$.
\end{proof}

The proposition is a consequence of the cited results. It gives neither an explicit function of $B$ nor a bound
independent of $B$, and it does not cover the symmetry strata, where \cref{prop:fixed-height-twists} shows that the
rational forms must be treated separately. A bound independent of $B$ on $U$ would follow from estimates
$r(P)+1\le A(1+\log H(P))$ and $m(P)\ge a(1+\log H(P))$ with constants $A,a>0$; neither is established here.

\subsection{Orbit searches}
By \cref{lem:unique-rational-form}, the fine points of $U$ of height at most $B$ give a finite list of rational
conjugacy classes, and the census provides it. A search would construct a model by \cref{prop:explicit}, minimise the
resultant as in step (2) of \cref{alg:small}, keep the reduction within the minimal resultant locus (the final model must
be resultant-minimal, which \cref{ex:nonminimal} shows is not automatic for a model of small naive height), and iterate
rational starting points exactly. Such searches give lower bounds for $M_3(B)$; a finite search gives no upper bound.

\section{Conclusion}

The rational points of the moduli space $\M_3$ of cubic rational maps split into fine and coarse points, and on the
generic stratum $U$ the splitting is governed by one quaternion algebra, $(-\mm,-2\xi_0)$, built from two invariants of even weight: the
discriminant of the divergence form and a principal minor of the Gram matrix of three quadratic covariants. Its
arithmetic follows the pattern of genus two, with the two loci of maps with an involution playing the role of the
bielliptic locus, but with a geometric twist: these loci have codimension two, the obstruction class has no
ramification divisor on $\M_3$, and the ramification reappears only on the exceptional divisors of blow-ups along
loci of transverse $A_1$-singularities. Local densities predict that the locus $\Lo2$ behaves like a divisor, while reduction into
$\Lo1$ is rare, and so the fine proportion should decay like $(\log B)^{-1/2}$, as in genus two. The census of all
\Nseven{} points of height at most $7$ agrees with the mechanism without exception and shows it at work on every
shell, with the fine proportion falling to $0.680$; it also shows how far the asymptotic regime is: at height $7$ a
fifth of the points still lie on sheets where the divergence has rational roots and every point is fine.

As in genus two, the obstruction lives at primes that vary from point to point; but here reciprocity transports it to
the fixed places $\infty$, $2$, $3$, where $83\%$ of the coarse points are obstructed. The second stage, the
construction of minimal models, is served by an algorithm: the explicit descent of \cref{prop:explicit} with
minimisation, reduction and a search off the minimal resultant locus recovers the exact minimal naive height, from the
moduli point alone, in all but one of the $7\,409$ test cases in which it is known. The same construction, with resultant
minimality kept, orders the cubic maps for the conjecture of \cite[Conj.~3.47]{silverman-2007-ads} on integers in
orbits, whose cubic case we have stated through the moduli height.

Several questions remain open: a proof of \cref{conj:density}, which requires a sieve on the rational points of the double
cover $\Xb$; the asymptotics of $N(B)$, conjecturally of order $B^8\log B$; and an analysis of the strata with automorphisms
and of the prime $2$, which, as the place that carries most of the reciprocity, deserves a description by congruence
conditions as complete as the one for the odd primes. For integer orbits, explicit bounds for the two quantities of
\cref{prop:generic-orbit-bound} in terms of the moduli height would make the finiteness of \cref{sec:orbits} effective.

\bibliographystyle{amsplain}
\bibliography{sh-136}

\end{document}